\documentclass[notitlepage,11pt,reqno]{amsart}
\usepackage[foot]{amsaddr}
\usepackage[numbers,sort]{natbib}
\usepackage[T1]{fontenc}
\usepackage{amssymb,nicefrac,bm,upgreek,mathtools,verbatim}
\usepackage[final]{hyperref}
\usepackage[mathscr]{eucal}
\usepackage{dsfont}

\usepackage{color}
\definecolor{dmagenta}{rgb}{.4,.1,.5}
\definecolor{dblue}{rgb}{.0,.0,.5}
\definecolor{mblue}{rgb}{.0,.0,.8}
\definecolor{ddblue}{rgb}{.0,.0,.4}
\definecolor{dred}{rgb}{.6,.0,.0}
\definecolor{dgreen}{rgb}{.0,.5,.0}
\definecolor{Eeom}{rgb}{.0,.0,.5}

\usepackage[normalem]{ulem}

\usepackage[margin=1in]{geometry}
\allowdisplaybreaks

\newtheorem{lemma}{Lemma}[section]
\newtheorem{theorem}{Theorem}[section]

\newtheorem{corollary}{Corollary}[section]

\theoremstyle{definition}
\newtheorem{definition}{Definition}[section]

\numberwithin{equation}{section}

\hypersetup{
  colorlinks=true,
  citecolor=dblue,
  linkcolor=dblue,
  frenchlinks=false,
  pdfborder={0 0 0},
  naturalnames=false,
  hypertexnames=false,
  breaklinks}
\usepackage[capitalize,nameinlink]{cleveref}

\crefname{section}{Section}{Sections}
\crefname{subsection}{Subsection}{Subsections}
\crefname{condition}{Condition}{Conditions}
\crefname{hypothesis}{Hypothesis}{Conditions}
\crefname{assumption}{Assumption}{Assumptions}
\crefname{lemma}{Lemma}{Lemmas}
\crefname{claim}{Claim}{Claims}

\Crefname{figure}{Figure}{Figures}

\crefformat{equation}{\textup{#2(#1)#3}}
\crefrangeformat{equation}{\textup{#3(#1)#4--#5(#2)#6}}
\crefmultiformat{equation}{\textup{#2(#1)#3}}{ and \textup{#2(#1)#3}}
{, \textup{#2(#1)#3}}{, and \textup{#2(#1)#3}}
\crefrangemultiformat{equation}{\textup{#3(#1)#4--#5(#2)#6}}%
{ and \textup{#3(#1)#4--#5(#2)#6}}{, \textup{#3(#1)#4--#5(#2)#6}}%
{, and \textup{#3(#1)#4--#5(#2)#6}}

\Crefformat{equation}{#2Equation~\textup{(#1)}#3}
\Crefrangeformat{equation}{Equations~\textup{#3(#1)#4--#5(#2)#6}}
\Crefmultiformat{equation}{Equations~\textup{#2(#1)#3}}{ and \textup{#2(#1)#3}}
{, \textup{#2(#1)#3}}{, and \textup{#2(#1)#3}}
\Crefrangemultiformat{equation}{Equations~\textup{#3(#1)#4--#5(#2)#6}}%
{ and \textup{#3(#1)#4--#5(#2)#6}}{, \textup{#3(#1)#4--#5(#2)#6}}%
{, and \textup{#3(#1)#4--#5(#2)#6}}

\crefdefaultlabelformat{#2\textup{#1}#3}

\makeatletter
\DeclareRobustCommand\widecheck[1]{{\mathpalette\@widecheck{#1}}}
\def\@widecheck#1#2{%
    \setbox\z@\hbox{\m@th$#1#2$}%
    \setbox\tw@\hbox{\m@th$#1%
       \widehat{%
          \vrule\@width\z@\@height\ht\z@
          \vrule\@height\z@\@width\wd\z@}$}%
    \dp\tw@-\ht\z@
    \@tempdima\ht\z@ \advance\@tempdima2\ht\tw@ \divide\@tempdima\thr@@
    \setbox\tw@\hbox{%
       \raise\@tempdima\hbox{\scalebox{1}[-1]{\lower\@tempdima\box
\tw@}}}%
    {\ooalign{\box\tw@ \cr \box\z@}}}

\def\subsection{\@startsection{subsection}{0}%
\z@{\linespacing\@plus\linespacing}{\linespacing}%
{\bf}}

\makeatother

\newcommand{\df}{\coloneqq}
\DeclareMathOperator{\Exp}{\mathbb{E}} 
\DeclareMathOperator{\Prob}{\mathbb{P}} 
\newcommand{\D}{\mathrm{d}}          
\newcommand{\E}{\mathrm{e}}          
\newcommand{\RR}{\mathbb{R}}         
\newcommand{\Rd}{{\mathbb{R}^d}}       
\newcommand{\NN}{\mathbb{N}}         
\newcommand{\ZZ}{\mathbb{Z}}         
\newcommand{\Ind}{\mathds{1}}            

\newcommand{\grad}{\nabla}

\newcommand{\cA}{\mathcal{A}}

\newcommand{\cC}{\mathcal{C}}     

\newcommand{\cG}{\mathcal{G}}

\newcommand{\cB}{\mathcal{B}}

\newcommand{\cK}{\mathcal{K}}

\newcommand{\cP}{\mathcal{P}}     
\newcommand{\sR}{\mathscr{R}}
\newcommand{\cT}{\mathcal{T}}

\newcommand{\abs}[1]{\lvert#1\rvert}
\newcommand{\norm}[1]{\lVert#1\rVert}

\DeclareMathOperator{\diam}{diam}

\DeclareMathOperator{\trace}{trace}
\DeclareMathOperator{\dist}{dist}

\begin{document}

\title[Mixed local-nonlocal operators]%
{\sc \textbf{Mixed local-nonlocal operators: maximum principles,
eigenvalue problems and their applications}}

\author{Anup Biswas and Mitesh Modasiya}

\address{
Department of Mathematics, Indian Institute of Science Education and Research, Dr. Homi Bhabha Road,
Pune 411008, India. Email: anup@iiserpune.ac.in; mitesh.modasiya@students.iiserpune.ac.in }

\date{}


\begin{abstract}
In this article we consider a class of non-degenerate elliptic operators 
obtained by superpositioning the Laplacian and a general nonlocal operator.
We study the existence-uniqueness results for Dirichlet boundary
value problems, maximum principles and generalized eigenvalue problems.
As applications to these results, we obtain Faber-Krahn inequality and
a one-dimensional symmetry result related to the Gibbons' conjecture.
The latter results substantially extend the recent results of 
Biagi et.\ al. \cite{BVDV,BDVV} who consider the operators of the form
$-\Delta + (-\Delta)^s$ with $s\in (0, 1)$.
\end{abstract}

\maketitle

\section{\bf Introduction}
In this article we consider an operator in $\Rd$ which is a combination of a local and 
a nonlocal operator. In particular, 
we consider operators of the form
\begin{equation}\label{Opt}
L u = \Delta u + \int_{\Rd} (u(x+y)-u(x)-\Ind_{\{|y|\leq 1\}} y\cdot \grad u(x)) j(y) \D{y},
\end{equation}
where $j:\Rd\setminus\{0\}\to [0, \infty)$ is a jump kernel satisfying 
\begin{equation}\label{kernel}
\int_{\Rd} (1\wedge |y|^2) j(y) \D{y}<\infty.
\end{equation}
Operators of the form \eqref{Opt} appears
naturally in the study of L\'{e}vy process. More precisely, the generator 
of a $d$-dimensional L\'evy process is given by the following general
structure
$$\cA u = \trace (a D^2u) + b\cdot \grad u + \int_{\Rd} (u(x+y)-u(x)-\Ind_{\{|y|\leq 1\}} y\cdot \grad u(x))  \nu(\D{y}),$$
where $a$ is a non-negative definite matrix, $b\in\Rd$ and $\nu$ is a L\'evy
measure satisfying 
$$\int_{\Rd} (1\wedge \abs{y}^2)\, \nu(\D{y})<\infty.$$
 The local elliptic operator corresponds to $\nu=0$. For 
$\nu(\D{y}) =|y|^{-d-2s} \D{y}$, the nonlocal part corresponds to the 
well-studied fractional Laplacian. In this article we
set $a=I, b=0$ and $\nu(\D{y})=j(y)\D{y}$ where $j$
satisfies \eqref{kernel}.
Let us denote by 
\begin{equation}\label{LK}
\psi(z)=\int_{\Rd}(1-e^{i z\cdot\xi}+\Ind_{\{|z|\leq 1\}} iz\cdot \xi)\, j(\xi)\D{\xi}.
\end{equation}
Let $Y$ be a pure-jump L\'{e}vy process with L\'evy-Khinchine exponent 
given by $\psi$ and $B$ be a Brownian motion, independent of $Y$, 
running twice as fast as the standard $d$-dimensional Brownian motion.
Let $X=B+Y$. We also assume that all the processes are defined
on a complete probability space $(\Omega, \mathcal{F}, \Prob)$.
It is well-known that $X$ is a strong Markov process and the semigroup
generated by $X$ is determined by the generator \eqref{Opt}. Furthermore,
the L\'evy-Khinchine representation of $X$ is given by
$$\Exp[e^{iz\cdot X_t}]=e^{-t(|z|^2+\psi(z))}\quad \text{for all}\; z\in\Rd\quad \text{and} \; t>0\,,$$
where $\Exp[\cdot]$ denotes the expectation with respect to the measure
$\Prob$.
For more details on this topic we refer to the book of Sato \cite{Sato}.
We impose the following assumption on $\psi$.
\begin{enumerate}
\item[(\hypertarget{A1}{\bf A1})]For some constant $C>0$ we have 
$|{\rm Im} (\psi(p))|\leq C (|p|^2 + |{\rm Re} (\psi(p))|)$ for all $p$ and
for all $r>0$ we have $\sup_{|p|\leq r} (|p|^2 + {\rm Re}(\psi(p)))>0$.
\end{enumerate}
It is easy to see that for $j$ symmetric (that is, $j(y)=j(-y)$) we have
$$\psi(z)=\int_{\Rd}(1-\cos(z\cdot \xi)) j(\xi)\D{\xi}\geq 0,$$
and thus, \hyperlink{A1}{(A1)} holds.

In this article we are concerned with equations of the form
\begin{equation}\label{Eq1}
L u =f(u, x)\quad \text{in}\; D, \quad \text{and}\quad 
u = g\quad \text{in}\; D^c.
\end{equation}
Integro-differential operators such as \eqref{Eq1} became quite popular
very recently. There is a large body of works dealing with elliptic operators with both local and nonlocal parts. But most of the works
restricted the nonlocal term to be the factional Laplacian 
\cite{AC20,Barles12,Barles08,BDVV,BDVV-a,BVDV,CVW,dEJ,RS15,BMS22}. However, 
there are many practical situations;
for instance, in biology \cite{CPSZ,N12}, mathematical finance 
\cite{OS07,BJK},
where the L\'evy measure need not be of 
fractional Laplacian type.  This gives us a motivation to consider an
integro-differential equation with a general L\'evy measure. By a solution we shall always mean a viscosity solution in the sense of Caffarelli and Silvestre \cite{CS09}.

\begin{definition}[Viscosity solution] \label{Defi1.1}
A bounded function $u:\Rd\to \RR$, upper (lower) semi-continuous in a domain $\bar{D}$, is said to be a viscosity subsolution (supersolution) to
$$L u = f\quad \text{in}\; D,$$
written as $Lu\geq f$ ($Lu\leq f$), if for any point $x\in D$ and a neighbourhood $N_x$ of $x$ in $D$, there exists a function $\varphi\in C^2(\overline{N_x})$ so that
$\varphi-u$ attains minimum (maximum) $0$ in $N_x$ at the point $x$, then letting
\[
v(y):=\left\{
\begin{array}{ll}
\varphi(y) & \text{for}\; y\in N_x,
\\[2mm]
u(y) & \text{otherwise},
\end{array}
\right.
\]
we have $L v(x)\geq f(x)$ ($L v(x)\leq f$, respectively). We say $u$ is a viscosity solution if it
is both sub and supersolution.
\end{definition}
We can also restrict ourselves to the test functions attaining strict
minimum or maximum in the definition above.
Our first result establishes the existence and uniqueness of solution.\begin{theorem}\label{T1.1}
Assume \hyperlink{A1}{(A1)}.
Let $D$ be an open, bounded Lipschitz domain in $\Rd$. Also, assume that 
$f\in\cC(\bar{D})$ and $g\in\cC_b(D^c)$. Then there exists a 
unique viscosity
solution $u\in\cC_b(\Rd)$ to
\begin{equation}\label{E1.5}
Lu = -f \quad \text{in}\; D, \quad \text{and}\quad 
u=g \quad \text{in}\; D^c.
\end{equation}
Furthermore, the unique solution can be written as
\begin{equation}\label{E1.6}
u(x)=\Exp_x\left[\int_0^{\uptau} f(X_t)\, \D{t}\right] + \Exp_x[g(X_\uptau)], \qquad x\in D,
\end{equation}
where $\uptau=\uptau_D$ denotes the first
 exit time of $X$ from $D$, that is,
$$\uptau_D=\inf\{t>0\; :\; X_t\notin D\}.$$
\end{theorem}
$\Exp_x[\cdot]$ denotes the conditional expectation operator conditioned on $X_0=x$, that is, $\Exp_x[\cdot]$ is the expectation operator with respect to the 
law of the L\'evy process $X+x$ where $X_0=0$.
When $f\in\cC^\alpha(\bar{D})$ and $g\in\cC^{2+\alpha}(D^c)$ for some
$\alpha>0$, the existence
of a unique classical solution to \eqref{E1.5} is known from
the work of Garroni and Menaldi \cite{GM02}.
In \cite{Barles08} Barles, Chasseigne and Imbert establish the existence of 
a viscosity solutions for a large class of nonlinear integro-differential operators. Unlike ours, \cite{Barles08} (see also \cite{Barles08a}) requires the operators to be strictly monotone in the zeroth order term. For the proof of \cref{T1.1}
we follow the approach of Cabr\'e-Caffarelli \cite{CC} and
Caffarelli-Silvestre \cite{CS09}. The following result plays a
key role in the proof of \cref{T1.1} 
and many other proofs in this article (compare it with \cite[Theorem~5.9]{CS09}).

\begin{theorem}\label{T1.2}
Let $D$ be an open bounded set,  $u$ and $v$ be two bounded functions such that u is upper-semicontinuous and v is lower-semicontinuous in $\overline{D}$. Also, assume that $L u \geq f$ and $L v \leq g$ in the viscosity sense in $D$, for two continuous functions $f$ and $g$. Then 
$L (u-v) \geq f-g$ in $D$ in the viscosity sense.

Furthermore, for a bounded function $u$ which is upper-semicontinuous in $\overline{D}$ and satisfies $L u \geq 0$ in $D$, we have  $\sup_{D} u \leq  \sup_{D^c} u$ .
\end{theorem}
Proof of \cref{T1.1} can be found in \cref{S-EU} whereas the proof 
of \cref{T1.2} follows from \cref{T5.1,T5.2}.
The stochastic representation \eqref{E1.6} of $u$ plays a key role in this
article. In particular, using this representation of the solution we can establish an Alexandrov-Bakelman-Pucci (ABP) maximum
principle.
\begin{theorem}[ABP-maximum principle]\label{T1.3}
Assume \hyperlink{A1}{(A1)} and let $D$ be an open, bounded set. Let $f:D\to \RR$ be continuous and $u\in \cC_b(\Rd)$ be a
viscosity subsolution to 
$$L u = -f\quad \text{in}\; \{u>0\}\cap D, \quad
\text{and}\quad u\leq 0\quad \text{in}\; D^{c}.$$
Then for every $p>\frac{d}{2}$, there exists a constant $C=C(d, p, \diam(D))$, satisfying
$$\sup_D u \leq C\, \norm{f^+}_{L^p(D)}.$$
\end{theorem}
In \cite[Theorem~3.2]{MS18} Mou and \'{S}wi\k{e}ch consider the Pucci extremal operators
and establish the ABP estimate for strong solutions. Similar estimate for
viscosity solutions can be found in Mou \cite{Mou19}. It should be observed that
the ABP estimates in \cite{Mou19,MS18} holds for $p>p_0$ where 
$p_0$ is some number in $[d/2, d)$. \cref{T1.3} shows that we can choose
$p_0=d/2$ for $L$. Also, compare this result with \cite[Theorem~1.9]{Cab}.
For a proof of \cref{T1.3}, see \cref{S-ABP}. Recently, Sobolev
regularity and maximum principles for the operator $-\Delta + (-\Delta)^s$
, $s\in (0, 1)$, are studied by Biagi et.\ al. in \cite{BDVV-a}. 
We also mention the work of Alibaud et.\ al. \cite{ADEJ} where
the authors provide a complete characterization of the translation-invariant integro-differential operators that satisfy the Liouville property
in the whole space.

In view of \cref{T1.1,T1.3} we can define a generalized principal 
eigenvalue for $L$ in the spirit of Berestycki, Nirenberg and Varadhan
\cite{BNV}. By $\cC_{+}(D)$ ($\cC_{b, +}(\Rd)$)
 we denote the set of all positive (bounded and non-negative)
continuous functions in $D$ (in $\Rd$, respectively).
Given any bounded domain $D$, the (Dirichlet) generalized principal eigenvalue of $L$ in $D$ is defined to be
\begin{equation}\label{P-eigen}
\lambda_D=\sup\{\lambda\; :\; \exists\; v\in\cC_{+}(D)\cap\cC_{b, +}(\Rd)\;
\text{satisfying}\; Lv  + cv + \lambda v\leq 0 \; \text{in}\; D\},
\end{equation}
where $c\in \cC_b(D)$.
We prove the existence of a unique eigenfunction.
\begin{theorem}\label{T1.4}
Grant \hyperlink{A1}{(A1)}. Let $D$ be a bounded
Lipschitz domain satisfying a uniform exterior sphere condition. Let $c\in\cC(\bar{D})$.
 There exists a unique $\psi_D \in \cC_b(\Rd)$  satisfying
\begin{align*}
L \psi_D + c\, \psi_D &= -\lambda_D \psi_D  \quad \text{in}\; D, \\
\psi_D &= 0 \quad \text{in}\;  D^c, \\
\psi_D &>0 \quad \text{in} \; D, \quad \psi_D (0) = 1.
\end{align*}
Moreover, if $u \in \cC_{b,+}(\Rd)$ is positive in $D$ and satisfies
$$
L u + cu \leq -\lambda u \quad \text{in} \quad D,
$$
for some $\lambda \in \RR$ then $\lambda \leq \lambda_D$. Furthermore, if $\lambda = \lambda_D$ and $u=0$ in $D^c$, then
we have $u = k \psi_D$ for some $k >0$. Furthermore, $\lambda_D$ is
the only Dirichlet eigenvalue with a positive eigenfunction.
\end{theorem}
For some recent works dealing with generalized eigenvalue problems of 
integro-differential operator we refer \cite{B20,BL21,QSX}. The proof 
of \cref{T1.4} is quite standard which uses Kre{\u\i}n-Rutman theorem and
a narrow domain maximum principle (\cref{C2.1}). For a proof
see \cref{T3.1}. 

We next concentrate on the Faber-Krahn inequality for the operator $L$.
In particular, we prove
\begin{theorem}[Faber-Krahn inequality]\label{T1.5}
Assume \hyperlink{A1}{(A1)}. 
Also, assume that $j(y)=j(|y|)$ and $j$ is radially decreasing.  Then for any bounded, open set $D$ with $|\partial D|=0$ we have
\begin{equation}\label{ET1.5A}
\lambda_D\geq \lambda_B,
\end{equation}
where $B$ is ball around $0$ satisfying $|B|=|D|$.
\end{theorem}
As is well-known Faber-Krahn inequality was first proved independently by Faber \cite{Faber} and Krahn \cite{Krahn} for the Laplacian. See also \cite[Chapter~2]{Hen06}. Very recently, Biagi et.\ al.\ \cite{BDVV} establish Faber-Krahn
inequality for the operator 
$-\Delta + (-\Delta)^s$ for $s\in(0,1)$. Their method uses  Schwarz symmetrization combined with the Polya-Szeg\"{o} inequality and 
\cite[Theorem~A.1]{FS08}. Since the inequality in \cite[Theorem~A.1]{FS08}
holds for a more general class of kernel $j$, it might be possible to
mimic the proof of \cite{BDVV} in an appropriate variational set-up
and Sobolev space to establish \eqref{ET1.5A}. However, our viscosity solution approach does not impose any additional regularity on the solution. Using \cref{T1.4}, we find
a probabilistic 
representation of the principal eigenvalue which together
with the Brascamp-Lieb-Luttinger
inequality gives us \cref{T1.5}. Also, note that our condition on $j$ is very general and our proof works in dimension one.
Proof of \cref{T1.5} can be found in \cref{S-FK}.

As another application of ABP maximum principle and \cref{T1.2} we also study symmetry properties of the positive solutions of semilinear equations.
Thanks to \cref{T1.2}, symmetry of the positive solutions can be established using the standard method
of moving plane \cite{GNN,BL21,FW14}. See \cref{S-Gib} for more
detail. Another interesting application of \cref{T1.2} is the
one-dimensional symmetry result related to the Gibbons' conjecture.
More precisely, we prove the following
\begin{theorem}\label{T1.6}
Assume \hyperlink{A1}{(A1)}.
Let $u\in\cC_b(\Rd)$ solve
\begin{equation*}
 \begin{split}
L u(x) &= f(u(x))\quad  \text{for}\; x \in \Rd,
\\[2mm]
\lim_{x_n \rightarrow \pm \infty } u(x^{\prime}, x_n) &= \pm 1 \quad \text{uniformly for } \;  x^{\prime} \in \RR^{d-1},
\end{split}
\end{equation*}
where $f\in\cC^1(\RR)$ satisfying
\begin{equation*}
\inf_{|r|\geq 1}  f^{'}(r) >0\,.
\end{equation*}
Then there exists a strictly increasing function $u_0:\RR\to\RR$ satisfying
$$u(y, t)=u_0(t)\quad \text{for all}\; y\in\RR^{d-1}, \; t\in\RR.$$
\end{theorem}
The above problem is inspired by a conjecture of G.\ W.\ Gibbons \cite{GT99}
which was formulated for the classical Laplacian operator. The classical 
Gibbons' conjecture was proved by several researchers using different
approaches; see for instance, \cite{BBG,BHM,F99}. In \cite{FV11} Farina and
Valdinoci prescribed a unified approach to this problem which also works
for several other classes of operators. Using the approach of \cite{FV11},
a similar problem is treated in \cite{BVDV} for the operator
$-\Delta+(-\Delta)^s$ with $s\in (0,1)$. For the proof of 
\cref{T1.6} we also broadly follow the approach of \cite{FV11} but,
 thanks to \cref{T1.2}, we do not impose any additional regularity assumption on $u$. Proof of \cref{T1.6} can be found in \cref{S-Gib}.
 
The rest of the paper is organized as follows. In \cref{S-ABP} we prove the
ABP maximum principle and the Hopf's lemma. The
existence of principal eigenfunction and Faber-Krahn inequality are
discussed in \cref{S-FK}. \cref{S-Gib} contains the proof of Gibbons'
conjecture whereas \cref{S-EU} provides the proofs of the existence and
uniqueness results for Dirichlet boundary value problems.

\section{Maximum principles}\label{S-ABP}
In this section we prove a Alexandrov-Bakelman-Pucci (ABP) maximum
principle (\cref{T2.1}), a Hopf's lemma (\cref{T2.2}) and a strong maximum principle (\cref{T2.3}).
We begin with the following estimate on the exit time.
\begin{lemma}\label{L2.1}
Assume \hyperlink{A1}{(A1)}.
Let $D \subset \Rd$ be a bounded domain.  For every $k \in N$ we have
$$
\sup_{x \in D} \Exp_x[ \uptau^ k] \leq  k! \left( \sup_{x \in D} \Exp_x[ \uptau] \right)^k.
$$
Moreover, there exists a constant $\theta = \theta(d, {\rm diam}(D))$, monotonically increasing with respect to $\diam(D)$, such that
$$
\sup_{x\in D}  \Exp_x[ \uptau^ k] \leq  k! \theta^k.
$$
\end{lemma}

\begin{proof}
Proof follows from \cite[Lemma~3.1]{BL17b}, \hyperlink{A1}{(A1)} and \cite[Remark~4.8]{Sch98}.
\end{proof}

Using \cref{L2.1} we find an ABP type estimate for the semigroup subsolutions. This is the content of our next lemma.
\begin{lemma}\label{L2.2}
Let $D$ be any bounded domain and $u:\Rd\to \RR$ be a bounded 
function satisfying
$$
u(x)\leq \Exp_x[u(X_\uptau)] + \Exp_x\left[\int_0^\uptau f(X_t)\, \D{t}\right]\quad \text{for all}\; x\in D,
$$
with $f \in L^p(D)$, for some $p > \frac{d}{2}$,
where $\uptau$ denotes the first exit time from $D$. Then there exists a constant $C_1 = C_1(p, d, \diam(D))$ such that
$$
\sup_D u^{+} \leq  \sup_{D^c}  u^{+} + C_1\,  \norm{f}_{L^p(D)}.
$$
\end{lemma}

\begin{proof}
For simplicity of notation we extend $f$ by zero outside of $D$. 
From the given condition it is easily seen that
$$ 
u(x) \leq \sup_{D^c}u^{+} + \Exp_x \left[ \int_0^{ \uptau} \vert f(X_s) \vert \D{s} \right], \quad x\in D.
$$
Thus, we only need to estimate the rightmost term in above expression.
Recall that $X=B+Y$ where $B$ is a $d$-dimensional Brownian
motion, running twice as fast as standard Brownian motion, and is independent of $Y$. Let $\nu_t$ be the transition probability of $Y_t$, starting from $0$,
that is,
$$\nu_t(A)=\Prob(Y_t\in A),\quad A\in\cB(\Rd).$$
Let
$$p^B_t(y)=(4\pi t)^{-\nicefrac{d}{2}}e^{-\frac{|y|^2}{4t}} ,$$
be the transition density of $B_t$ starting from $0$. Then the transition
density of $X_t$, starting from $x$, is given by
$$p_t(x, y)= \int_{\Rd} p^B_t(z-x-y)\nu_t(\D{z}),\quad t>0.$$
In particular, for any $t \in ( 0 ,  \infty)$ and $x, y \in \Rd$, we have 
\begin{align}\label{EL2.2A}
p_t(x, y) &= \int_{\Rd} \frac{1}{(4 \pi t)^{\frac{d}{2}}} 
\E^{- \frac{|x-y-z|^2}{4t}} \nu_t(\D{z})  \nonumber
\\
&\leq \frac{1}{(4 \pi t)^{\frac{d}{2}}} \nu_t(\Rd) = \frac{1}{(4 \pi t)^{\frac{d}{2}}} \; . 
\end{align}
Next we write
\begin{align}\label{EL2.2B}
\Exp_x \left[ \int_0^{ \uptau} \vert f(X_s) \vert \D{s} \right] &= \Exp_x \left[ \int_0^{ \infty} \Ind_{ \left\lbrace \uptau > s\right\rbrace } \vert f(X_s) \vert \D{s} \right] \nonumber
\\
&\leq \Exp_x \left[ \int_0^{ 1}  \vert f(X_s) \vert \D{s} \right]  + \Exp_x \left[ \int_1^{ \infty} \Ind_{\left\lbrace \uptau > s\right\rbrace } \vert f(X_s) \vert \D{s} \right].
\end{align}
We estimate the first term on the rhs as 
\begin{align*}
\Exp_x \left[ \int_0^{ 1}  \vert f(X_s) \vert \D{s} \right] &=\int_0^1 \int_{\Rd} \vert f(y) \vert p_s(x, y)\D{y}\, \D{s} \leq \Vert f \Vert_{L^p(D)} \int_0^1  \Vert p_s(x, \cdot)\Vert_{p^{'}} \D{s}
 \\
&\leq \Vert f \Vert_{L^p(D)} \int_0^1 \left[ \int_{\Rd} ( p_s(x,  y))^{p^{\prime}} \D{y} \right]^{\frac{1}{p^{\prime}}} \D{s}
 \\
&\leq \frac{1}{(4\pi)^{\nicefrac{d}{2p}}}\norm{f}_{L^p(D)} \int_0^1 \left[ (s^{- \frac{d}{2}})^{p^{\prime}-1}  \int_{\Rd} ( p_s(x,  y)) \D{y} \right]^{\frac{1}{p^{\prime}}} \D{s}
 \\
&\leq  \frac{1}{(4\pi)^{\nicefrac{d}{2p}}} \norm{f}_{L^p(D)} \int_0^1  s^{-\frac{d}{2p}} \D{s} 
= \frac{1}{(4\pi)^{\nicefrac{d}{2p}}}\frac{2p}{2p -d}   \norm{f}_{L^p(D)},
\end{align*}
where in the third line we use \eqref{EL2.2A} and
$p, p^{\prime}$ are H\"{o}lder conjugates. To deal with the
rightmost term in \eqref{EL2.2B} we choose $k \in N$ with $k > p^{\prime}$.
Using \cref{L2.1} we then calculate
\begin{align*}
 \Exp_x \left[ \int_1^{ \infty} \Ind_{\left\lbrace \uptau > s\right\rbrace } \vert f(X_s) \vert \D{s} \right] &\leq \int_1^{\infty} (\Prob_x (\uptau > s))^{\frac{1}{p^{\prime}}} \Exp_x \left[ \vert f(X_s)^p \vert \right]^{\frac{1}{p}} \D{s}
 \\
 &\leq \frac{1}{(4\pi)^{\nicefrac{d}{2p}}} \norm{f}_{L^p(D)} \int_1^{\infty} (\Prob_x (\uptau > s))^{\frac{1}{p^{\prime}}} \D{s}
 \\
 &\leq \frac{1}{(4\pi)^{\nicefrac{d}{2p}}} \Vert f \Vert_{L^p(D)} \int_1^{\infty} s^{-\frac{k}{p^{\prime}}} \Exp_x \left[ \uptau^k \right]^{\frac{1}{p^{\prime}}} \D{s} 
 \\
 &\leq \frac{1}{(4 \pi)^{\frac{d}{2p}}} \norm{f}_{L^p(D)} 
 \frac{p^{\prime}}{k - p^{\prime}}   (k! \theta^k)^{\frac{1}{p^{\prime}}}.
\end{align*}
Combining these estimates in \eqref{EL2.2B} completes the proof.
\end{proof}

Using \cref{L2.1} and the arguments of \cite[Theorem~3.1]{BL21}
then gives us the ABP maximum principle.

\begin{theorem}[ABP-maximum principle]\label{T2.1}
Assume \hyperlink{A1}{(A1)}. Let $f:D\to \RR$ be continuous and $u\in \cC_b(\Rd)$ be a
viscosity subsolution to 
$$L u = -f\quad \text{in}\; \{u>0\}\cap D, \quad
\text{and}\quad u\leq 0\quad \text{in}\; D^c.$$
Then for every $p>\frac{d}{2}$, there exists a constant $C=C(d, p, \diam(D))$ satisfying
$$\sup_D u \leq C\, \norm{f^+}_{L^p(D)}.$$
\end{theorem}

As an application of \cref{T2.1} we obtain a narrow domain maximum principle.

\begin{corollary}[Maximum principle for narrow domains]\label{C2.1}
Assume \hyperlink{A1}{(A1)}. Let $u\in\cC_b(\Rd)$ be a viscosity subsolution to
$$L u + c\, u =0 \quad \text{in}\; D, \quad \text{and}\quad u\leq 0\quad 
\text{in}\; D^c,$$
for some continuous function $c: D\to \RR$. Then there exists a constant
$\varepsilon=\varepsilon(d,\norm{c}_{L^\infty (D)}, \diam(D))$ such that
$u\leq 0$ in $\Rd$, whenever $|D|<\varepsilon$. 
\end{corollary}

\begin{proof}
The result follows from \cref{T2.1} by choosing $f= \norm{c}_{L^\infty (D)} u^+$
in $\{u>0\}$.
\end{proof}

\begin{corollary}\label{C2.2}
Assume \hyperlink{A1}{(A1)} and let $D$ be a 
bounded open set. Let $u, v\in \cC_b(\Rd)$ satisfy
$$ L u + cu \geq f, \quad L v + cv \leq g \quad \text{in}\; D,$$
for some $c, f, g\in \cC(D)$. Also, assume that $c\leq 0$ and $f\geq g$ in
$D$. Then, if $u\leq v$ in $D^c$, we have $u\leq v$ in $\Rd$.
\end{corollary}

\begin{proof}
Using Theorem~\ref{T1.2} we get that
$(L+c)w\geq 0$ in $D$ with $w\leq 0$ in $D^c$
where $w=u-v$. 
Note that, moving $cw$ on the rhs, we can take $f=0$ on $\{w>0\}$ in Theorem~\ref{T2.1}.
Then applying Theorem~\ref{T2.1}, we obtain $w\leq 0$ in $\Rd$. Hence the proof.
\end{proof}

Next we prove the Hopf's lemma. At this point we mention a recent 
interesting work of Klimsiak and Komorowski \cite{KK21} where an 
abstract Hopf's type lemma is obtained for the semigroup solutions of a
general integro-differential operator. 

\begin{theorem}[Hopf's lemma]\label{T2.2}
Let $L u + cu \leq 0$ in $D$ where $u \in \cC_b(\Rd)$ and $c\in\cC_b(\bar{D})$.
Suppose that $u >0$ in $D$ and  non-negative in $\Rd$. Then there exists $ \eta >0 $ such that for any $x_0 \in\partial D$ with $u(x_0)=0$ we have 
$$
\frac{u(x)}{(r - \vert x-z \vert)} \geq \eta ,
$$
for all $x \in B_r(z) \cap B_{\frac{r}{2}}(x_0)$, where $B_r(z) \subset D$ is a ball that touches $\partial D$ at $x_0$.
\end{theorem}

\begin{proof}
 Since $u>0$ in $D$, without any loss of generality, we may assume that $c\leq 0$.
Let $K = B_r (z) \cap B_{\frac{r}{2}}(x_0)$ and define 
$$
v(x) = \E^{- \alpha q(x)} - \E^{- \alpha r^2} ,
$$
where $q(x) = \vert z - x \vert^2 \wedge 9r^2$. Clearly $v > 0$ in $B_r(z)$, $ v(x)= 0$ on $\partial B_r(z)$,  and $v \leq 0$ in $\Rd \setminus B_{r}(z)$.  For $x \in B_{2r}(z)$ we have
$$
\Delta v = \alpha \E^{- \alpha \vert x - z \vert^2} \left( 4\alpha \vert x - z \vert^2 - 2d \right).
$$
Fix any $x\in B_{2r}(z)$. Using the convexity of $x\mapsto \E^{x}$ we first note that, for $|y|\leq 1$,
\begin{align*}
& v(x+y)-v(x)-\Ind_{\{|y|\leq 1\}}y\cdot \grad v(x)
\\
&\qquad =\E^{-\alpha|x+y-z|^2}-\E^{-\alpha|x-z|^2}+2\alpha\Ind_{\{|y|\leq 1\}} y\cdot (x-z) \E^{-\alpha|x-z|^2}
\\
&\qquad\geq -\alpha \E^{-\alpha|x-z|^2}\left(|x+y-z|^2-|x-z|^2-2\Ind_{\{|y|\leq 1\}}
y\cdot (x-z)\right)  =   -\alpha \E^{-\alpha|x-z|^2} |y|^2 .
\end{align*}
Therefore, for $x\in B_{2r}(z)$, we have
\begin{align*}
&\int_{\Rd} (v(x+y)-v(x)-\Ind_{\{|y|\leq 1\}}y\cdot \grad v(x)) j(y)\D{y}
\\
&\qquad  =  \int_{|y| \leq 1} (v(x+y)-v(x)-\Ind_{\{|y|\leq 1\}}y\cdot \grad v(x)) j(y)\D{y} + \int_{|y|>1}  (v(x+y) - v(x))  j(y) \D{y}
\\
&\qquad \geq -\alpha \E^{-\alpha|x-z|^2} \int_{|y|\leq 1} |y|^2 j(y)\, \D{y} + \int_{|y|>1}  (v(x+y) - v(x))  j(y) \D{y}
\\
&\qquad \geq -\alpha \E^{-\alpha|x-z|^2} \int_{|y|\leq 1} |y|^2 j(y)\, \D{y}
+ \E^{-\alpha|x-z|^2} \int_{|y|>1} \left(\E^{-9\alpha r^2}-1\right)j(y)\, \D{y},
\end{align*}
where in the last line we used $|x-z+y|^2\wedge 9r^2\leq |x-z|^2
+ 9r^2$.
Thus, using \eqref{kernel} and $x \in B_{2r}(z)$, we obtain
\begin{align*}
&L v(x) + c(x) v(x)  
\\
&\quad \geq \alpha \E^{- \alpha \vert x-z \vert^2} \Bigl[ 4 \alpha \vert x-z \vert^2 - 2d - \int_{|y|\leq 1}|y|^2 j(y) \D{y} 
+ \alpha^{-1} \int_{|y|>1} \left(\E^{-9\alpha r^2}-1\right)j(y)\, \D{y} 
\\
&\hspace{4em} -\norm{c}_{L^\infty (D)} \alpha^{-1} ( 1 - \E^{- \alpha ( r^2 - \vert x-z \vert^2)}) \Bigr] .
\end{align*}
For $|x-z|\geq r/2$, we can choose $\alpha$ large enough so that
\begin{equation}\label{ET2.3A}
L v + c v>0 \quad \text{for}\;\; \frac{r}{2}\leq |x-z|<2r.
\end{equation}
Let $m = \min_{D^-_{\frac{r}{2}}}u$ where 
$D^-_{\frac{r}{2}}=\{y\in D\; :\; \dist(y, \partial D) \geq \frac{r}{2}\}$.
Defining $ w =  m v$,   we have $L(w-u) + c(w-u) \geq 0$ in $B_{r}(z) \setminus B_{r/2}(z)$, by  \cref{T1.2}, and $w-u \leq 0$ in $(B_{r}(z) \setminus B_{r/2}(z))^c$.  Thus as a consequence of Corollary~\ref{C2.2},  we obtain
\begin{align*}
u &\geq  m v 
\\
 &= m \E^{-\alpha r^2} (\E^{\alpha (r^2 - \vert x-z \vert^2)} -1)
 \\
&\geq m \E^{-\alpha r^2} \alpha (r^2 - \vert x-z \vert^2) .
\end{align*}
This completes the proof.
\end{proof}

As a by-product of the proof of Hopf's lemma above we obtain a strong maximum principle (compare with Ciomaga \cite{Ciomaga}).
\begin{theorem}[Strong maximum principle]\label{T2.3}
Assume \hyperlink{A1}{(A1)} and let $c\in\cC_b(D)$.
Let $L u + cu \leq 0$ in $D$ and $u \in \cC_b(\Rd)$ be non-negative in $\Rd$. Then either $u>0$ in $D$ or it is identically $0$ in $D$.
\end{theorem}
\begin{proof}
Since $u$ is non-negative,  without any loss of generality,  we may assume that $c\leq 0$.  If we take $K= \overline{ \{ u=0\} \cap D}$, then we want to show that $K \cap D$ is either empty or 
$D$.  Suppose, to the contrary, $K \cap D $ is non-empty and is not equal to the set $D$.   This means  $D \setminus K$ is also non-empty.   Hence we can find a point 
$z \in D\setminus K$ and $r$ small,  such that $x_0 \in \partial B_r(z)$ for some $x_0 \in K \cap D$,  $B_{2r}(z) \subset D$ and $u>0$ in $B_r(z)$.  \\
Now we consider the function $v$ that we constructed in \cref{T2.2}, that is 
$$
v(x) = \E^{- \alpha q(x)} - \E^{- \alpha r^2} ,
$$
where $q(x) = \vert z - x \vert^2 \wedge 9r^2$.  Again we have $v > 0$ in $B_r(z)$, $ v(x)= 0$ on $\partial B_r(z)$,  and $v \leq 0$ in $\Rd \setminus B_{r}(z)$.  Also by \eqref{ET2.3A} we have
\begin{equation}\label{ET2.3B}
L v + c v>0 \quad \text{in}\;\; B_{2r}(z) \setminus B_{r/2}(z).
\end{equation}
Let $m = \min_{B_{\frac{r}{2}}(z)}u$ and
define $ w =  m v$.  Then following similar argument as in \cref{T2.2} we have $u \geq w$ in $\Rd$.  Since $w \in C^2(B_{2r}(z))$ and $w(x_0) = u(x_0) =0$,  we use $w$ as a test function and define
\[
\phi(y):=\left\{
\begin{array}{ll}
w(y) & \text{for}\; y\in B_{r}(x_0),
\\[2mm]
u(y) & \text{otherwise}.
\end{array}
\right.
\] 
Then by the definition of viscosity supersolution we have 
$$L\phi(x_0) + c(x_0) \phi(x_0) \leq 0 .$$
Clearly, $w \leq \phi $ in $\Rd$, $\grad \phi(x_0) = \grad w(x_0) $ and  $\Delta \phi(x_0) = \Delta w(x_0) $. Thus  
$Lw(x_0) + c(x_0) w(x_0) \leq 0$ which contradicts the fact that $Lw(x_0) + c(x_0) w(x_0) > 0$ (see \eqref{ET2.3B}). Hence $K \cap D$ is either empty or $D$.   This completes the proof.
\end{proof}

\section{Generalized principal eigenvalue and Faber-Krahn inequality}\label{S-FK}
In this section we study the generalized eigenvalue problem of $L$
(\cref{T3.1}) and 
then we establish a Faber-Krahn inequality (\cref{T3.2}).
We begin with the following boundary estimate which will be useful.
\begin{lemma}\label{L3.1}
Assume \hyperlink{A1}{(A1)}.
Let $D$ be a bounded domain satisfying an uniform exterior sphere
condition with radius $r>0$.  Let $u\in\cC_b(\Rd)$ be a viscosity solution to 
$$
L u = f \quad in \quad D, \quad \quad u = 0 \quad in \quad D^c,
$$
for some $f\in L^\infty(D)$.
Then there exists a constant $C$, dependent on $r, d, \diam(D)$, satisfying
$$
\abs{u(x)} \leq C  \norm{f}_{L^\infty(D)}\, \dist(x , \partial D)
\quad \text{for}\; x\in D.
$$
\end{lemma}

\begin{proof}
Let $B$ be ball containing $D$. Then $v(x)=\Exp_x[\uptau_B]$ solves 
(see \cref{T1.1})
$$L u =-1\quad \text{in}\; B, \quad u=0\quad \text{in}\; B^c.$$
Applying comparison principle, \cref{T1.2}, it then follows that
\begin{equation}\label{EL3.1A}
|u(x)|\leq \norm{f}_{L^\infty(D)} v(x),\quad x\in\Rd.
\end{equation}
Without any loss of generality we may assume $r\in (0, 1)$.
By \cite[Lemma~5.4]{Mou19} there exists a bounded, Lipschitz continuous function
$\varphi$, with Lipschitz constant $r^{-1}$, satisfying
$$
\begin{cases}
\varphi = 0, & \quad \text{in} \quad \bar{B}_r , \\
\varphi > 0, & \quad \text{in} \quad \bar{B}_r^c , \\
\varphi \geq \varepsilon, & \quad \text{in} \quad  B_{(1+\delta)r}^c   , \\
L \varphi \leq -1 , & \quad \text{in} \quad B_{(1+\delta)r},
\end{cases}
$$
for some constant $\varepsilon, \delta$, where $B_r$ denotes the ball of
radius $r$ around $0$.
 Now for any point $y\in\partial D$
we can find another point $z\in D^c$ such that $\overline{B_r(z)}$ touches
$\partial D$ at $y$. Defining 
$w(x)=\varepsilon^{-1}\norm{f}_{L^\infty(D)} \norm{v}_{L^\infty (D)} \varphi(x-z)$ and using \eqref{EL3.1A}
it follows that $|u(x)|\leq w(x)$ in $B_{(1+\delta) r}(z)\cap D$, by
comparison principle. This relation holds for any $y\in \partial D$.
Now for any point $x\in D$ with $\dist(x, \partial D)<\delta r$ we can find
$y\in \partial D$ satisfying $\dist(x, \partial D)=\abs{x-y}<\delta r$. By previous estimate
we obtain
\begin{align*}
|u(x)| &\leq \varepsilon^{-1} \norm{f}_{L^\infty (D)} \norm{v}_{L^\infty (D)} \varphi(x-z)
\leq \varepsilon^{-1} \norm{f}_{L^\infty (D)} \norm{v}_{L^\infty (D)} (\varphi(x-z)-
\varphi(y-z)) \\
&\leq \varepsilon^{-1} \norm{f}_{L^\infty (D)}  \norm{v}_{L^\infty (D)} \,
r^{-1}\dist(x, \partial D).
\end{align*}
Now the proof follows from \eqref{EL3.1A}.
\end{proof}

Now fix a Lipschitz domain $D$ satisfying a uniform exterior sphere
condition.
In view of \cref{T1.1} we can define a map 
$\cT: \cC(\bar{D})\to \cC_0(D)$(the space of continuous function in $\bar{D}$ vanishing on the boundary) as
follows: $\cT[f]=u$ where $u$ is the unique viscosity solution to 
$$L u= -f \quad \text{in}\;  D, \quad \text{and}\quad
u=0\quad \text{in}\quad D^c.$$

\begin{lemma}\label{L3.2}
Under \hyperlink{A1}{(A1)},
the map $\cT$ is a bounded linear, compact operator.
\end{lemma}

\begin{proof}
It is evident that $\cT$ is a bounded linear operator. So we only need to show that $\cT$ is compact. Let $\cK$ be a bounded subset of $\cC(\bar{D})$,
that is, there exists a constant $\kappa$ such that $\norm{f}_\infty\leq \kappa$ for all $f\in\cK$. Define $\cG=\{u\; :\; u=\cT[f]\; \text{for some}\; f\in\cK\}$. By \cref{L3.1}, $\cG$ is bounded in $\cC_0(D)$. Let us now
show that $\cG$ is also equicontinuous. Consider $\varepsilon>0$.
For $\delta_1>0$ let us define
$$D^-_{\delta_1}=\{x\in D\; :\; \dist(x, \partial D)>\delta_1\}.$$
Using \cref{L3.1}, we can then choose $\delta_1>0$ small enough 
to satisfy
\begin{equation}\label{EL3.2A}
\sup_{D\setminus D^-_{2\delta_1}}\abs{u(x)}<\varepsilon/2
\quad \text{for all}\; u\in\cG.
\end{equation}
Again, by \cite[Theorem~4.2]{Mou19}, there exists $\alpha>0$ satisfying
\begin{equation}\label{EL3.2B}
\sup_{x\neq y, x, y\in D^{-}_{\delta_1}}\frac{|u(x)-u(y)|}{|x-y|^\alpha}
\leq \kappa_2\quad \text{for all}\; u\in\cG,
\end{equation}
for some constant $\kappa_2$. Choose $\delta\in (0, \delta_1)$ satisfying
$\kappa_2\delta^\alpha<\varepsilon$. Then, from \cref{EL3.2A,EL3.2B}, we obtain
$$|u(x)-u(y)|<\varepsilon\quad   \text{for all} \; u\in \cG, \; x, y\in D, \; |x-y| \leq \delta . $$
Therefore $\cG$ is equicountinuous.
Hence by Arzel\`{a}-Ascoli theorem we have $\cG$ compact, completing the
proof.
\end{proof}

From \cref{L3.2} and \cref{C2.2} we get the following existence result.
\begin{lemma}\label{L3.3}
Grant \hyperlink{A1}{(A1)}. Let $D$ be a bounded Lipschitz domain satisfying a uniform exterior sphere condition. Suppose $c\in\cC(\bar{D})$ with $c\leq 0$. Then
for any $f\in \cC(\bar{D})$ there exists a unique solution $u$ to
\begin{equation}\label{EL3.3A}
L u + cu = -f\quad \text{in}\; D, \quad \text{and}\quad u=0\quad \text{in}\;
D^c.
\end{equation}
\end{lemma}

\begin{proof}
Fixing $f \in \cC(\bar{D})$, we define a map $F : \cC_0(D)\rightarrow \cC_0(D)$ 
as
$$
u=F(v) = \cT[f+cv] ,
$$
where $\cT$ is same as in \cref{L3.2}. From \cref{L3.2} it follows that
$F$ is continuous and compact.
Consider a set
$$
E = \left\lbrace v \in \cC_0(D) : v = \lambda F(v) \quad \text{for some}\; \lambda \in \left[ 0 , 1 \right] \right\rbrace.
$$

Claim: $E$ is bounded in $\cC_0(D)$.

Suppose, to the contrary, that $E$ is not bounded.
Then there exist $v_n\in E$, for $n \in \NN$,  such that
$\norm{v_n}_{L^\infty (D)} \rightarrow \infty $ as $n \rightarrow \infty$.  So we have couples $(v_n , \lambda_n)$ satisfying $v_n = \lambda_n F(v_n)$ which gives us
$$
L v_n = \lambda_n (- f - cv_n)\quad \text{in}\; D.
$$
Letting $w_n = \norm{v_n}^{-1}_{L^\infty (D)} v_n$ we get from above that
$$
L w_n = \lambda_n \frac{ - f}{\norm{v_n}_{L^\infty (D)}} - c \lambda_n w_n.
$$
Since $\norm{w_n}_{L^\infty (D)} = 1$ for all $n$ we see that
$$
\Vert \lambda_n \frac{f}{\norm{v_n}_{L^\infty (D)}} + c \lambda_n w_n \Vert_{L^\infty (D)} \leq \kappa_2,
$$
for some $\kappa_2> 0$. Therefore, using \cref{L3.2},
$\{w_n\;:\, n\geq 1\}$ is equicontinuous and hence, up to a subsequence, $w_n \rightarrow w$  in $\cC_0(D)$. By the stability property of the viscosity solutions we obtain
\begin{align*}
L w = 0- c \lambda w  \quad \text{in}\; D \quad
\text{and}\quad  w= 0 \quad \text{in}\; D^c,
\end{align*}
for some $\lambda \in [ 0, 1]$.  Hence $w$ solves the equation
\begin{equation*}
\begin{split}
L w + c \lambda w&= 0  \quad \text{in}\; D,
\\
w &= 0 \quad \text{in}\; D^c .
\end{split}
\end{equation*}
From \cref{C2.2} we see that $w=0$ in $\Rd$.
But this contradicts the fact that
$\norm{w}_{L^\infty (D)} =1$, and this proves our claim. 


Applying Leray-Schauder theorem we must have a fixed point of $F$. This
gives the existence of solution for \eqref{EL3.3A}. Uniqueness again follows
from the above arguments.
\end{proof}

Now we are ready to prove the existence of principal eigenfunction.
\begin{theorem}\label{T3.1}
Assume \hyperlink{A1}{(A1)}. Let $D$ be a bounded domain satisfying a uniform exterior sphere condition. Let $c\in\cC(\bar{D})$.
 There exists a unique $\psi_D \in \cC_b(\Rd)$,  satisfying
\begin{align*}
L \psi_D + c \psi_D &= -\lambda_D \psi_D  \quad \text{in}\; D, \\
\psi_D &= 0 \quad \text{in}\;  D^c, \\
\psi_D &>0 \quad \text{in} \; D, \quad \psi_D (0) = 1.
\end{align*}
Moreover, if $u \in \cC_{b,+}(\Rd)$ is positive in $D$ and satisfies
$$
L u + cu \leq -\lambda u \quad \text{in} \quad D,
$$
for some $\lambda \in \RR$ then $\lambda \leq \lambda_D$. Furthermore, if $\lambda = \lambda_D$ and $u=0$ in $D^c$, then
we have $u = k \psi_D$ for some $k >0$. Furthermore, $\lambda_D$ is
the only Dirichlet eigenvalue with a positive eigenfunction.
\end{theorem}

\begin{proof}
The proof technique is quite standard and follows 
by combining Kre{\u\i}n-Rutman theorem with \cref{L3.2,L3.3}. Replacing $c$ by $c-\norm{c}_{L^\infty (D)}$ we can 
assume that $c\leq 0$. Using \cref{L3.3} we define a map $\cT_1:\cC_0(D)
\to\cC_0(D)$ as follows: $\cT_1[u]=v$ if and only if
$$L v + cv=-u\quad \text{in}\; D, \quad \text{and}\quad
v=0\quad \text{in}\; D^c.$$
Since, by comparison principle (see  \cref{C2.2}), $\norm{v}\leq \norm{u}\, \max\{\norm{w_+}, \norm{w_{-}}\}$ where
$$L w_{\pm} + c w_{\pm}=\pm 1\quad \text{in}\; D, \quad \text{and}\quad
w_{\pm}=0\quad \text{in}\; D^c,$$
it follows from \cref{L3.2} that $\cT_1$ is a compact, bounded linear map.
Again, if $u_1\leq u_2$ in $\cC_0(D)$, by comparison principle 
(see \cref{C2.2}) we get $\cT_1[u_1]\leq \cT_1[u_2]$. Furthermore, if
$u_1\lneq u_2$, then $\cT_1[u_1]< \cT_1[u_2]$ in $D$ by \cref{T2.3}. Let $f\gneq 0$ be nonzero compactly supported function in $D$. Then, for $\cT_1[f] = v$, we have $v>0$ in $D$ and therefore, we can find $M>0$ satisfying $M\cT_1[f]> f$
in $D$. Denote by $\cP$ the cone of non-negative functions in $\cC_0(D)$.
From comparison principle \cref{C2.2}, it is easily seen that $\cT_1(\cP)\subset \cP$. Therefore, 
Kre{\u\i}n-Rutman applies to $\cT_1$ and we find
$\lambda_D > 0$ and $\psi_D \in C_0(D)$ with $\psi_D>0$ in $D$ satisfying 

\begin{equation}\label{ET3.1A}
\begin{split}
L\psi_D + c\psi_D + \lambda_D \psi_D &= 0 \quad \text{in}\; D , 
\\
\psi_D &= 0 \quad \text{in} \; D^c .
\end{split}
\end{equation}
Now we focus on the second part of the theorem.
Consider a non-negative function $u\in\cC_b(\Rd)\cap \cC_+(D)$  satisfies
$$L u + c u + \lambda u\leq 0 \quad \text{in} \; D,$$
for some $\lambda\in\RR$. Suppose, to the contrary, that $\lambda> \lambda_D$. Then using \cref{C2.1}, Theorem~\ref{T2.3} and the proof
of \cite[Theorem~3.2]{BL21} we find $\mathfrak{z}>0$ satisfying $u=\mathfrak{z}\psi_D$ in $D$. Since minimum of two
viscosity supersolutions is also a supersolution, we have $\mathfrak{z}\psi_D = \min\{u, \mathfrak{z}\psi_D\}$ and
$$L (\mathfrak{z}\psi_D) + (c+\lambda) (\mathfrak{z}\psi_D) \leq 0 \quad \text{in} \; D.$$
Applying Theorem~\ref{T1.2}, we see that $(\lambda-\lambda_D)\psi_d\leq 0$ in $D$ which is a contradiction. Hence $\lambda\leq \lambda_D$.
Rest of the proof follows from \cite[Theorem~3.2]{BL21}.
\end{proof}

Our next aim is to prove Faber-Krahn inequality and to do so
we need certain continuity property of the principal eigenvalue 
with respect to the domains. To do so we need the following condition.

\begin{itemize}
\item[(\hypertarget{A4}{A4})]
The domain $D$ is Lipschitz and bounded with uniform exterior sphere condition of radius $r$. Furthermore, there exists a collection
of bounded, Lipschitz decreasing domains $\{D_n\}$ such that $\cap_{n} D_n=\bar{D}$
and each $D_n$ satisfies uniform exterior sphere condition of radius $r$.
\end{itemize}
It can be easily seen that convex domains, $\cC^{1,1}$ domains satisfy
the above condition.
In the next lemma we prove result on continuity of $\lambda_D$.
\begin{lemma}\label{L3.4}
Assume \hyperlink{A1}{(A1)} and \hyperlink{A4}{(A4)}. Denote by $\lambda_n=\lambda_{D_n}$. Then $\lambda_n\to \lambda_D$ as $n\to\infty$.
\end{lemma}

\begin{proof}
From \cref{T3.1} we notice that $\lambda_n\leq \lambda_{n+1}$.
Let $\lim_{n \rightarrow \infty}  \lambda_{_n} = \lambda$.  Evidently,
$ \lambda \leq \lambda_{ D}$. Using condition \hyperlink{A4}{(A4)},
\cref{L3.1} and the proof of
\cref{L3.2} it follows that $\{\psi_n\}$ is equicontinuous in $\Rd$
where $\psi_n$ is the principal eigenfunction corresponding to $\lambda_n$.
We also normalize $\psi_n$ to satisfy $\norm{\psi_n}_{L^\infty (\Rd)}=1$. Applying
Arzel\'{a}-Ascoli we can extract a subsequence of $\psi_n$ converging
to $\psi$ and by the stability property of the viscosity solution we
obtain
$$L \psi + (c+\lambda) \psi=0\quad \text{in}\; D, 
\quad \text{and}\quad \psi=0\quad \text{in}\; D^c.$$
Since $\psi\gneq 0$, by strong maximum principle Theorem~\ref{T2.3}, we must have $\psi>0$ in 
$D$. Then, by \cref{T3.1}, we must have $\lambda=\lambda_D$. Hence the 
proof.
\end{proof}

Next we find a representation of the principal eigenvalue which is
crucial for the proof of Faber-Krahn inequality.
\begin{lemma}\label{L3.5}
Consider the setting of \cref{L3.4} and let $c=0$. Let $\lambda_D$ be 
the corresponding principal eigenvalue. Then
\begin{equation}\label{EL3.5A}
\lambda_D = -\lim_{t\to\infty} \frac{1}{t}\log \Prob_x(\uptau>t)\quad 
\text{for all}\; x\in D.
\end{equation}
\end{lemma}

\begin{proof}
The principal eigenpair $(\psi_D, \lambda_D)$ satisfies
$$L\psi_D + \lambda_D\psi_D=0\quad \text{in}\; D,
\quad \text{and}\quad \psi_D=0\quad \text{in}\; D^c.$$
From the arguments of \cite[Lemma~3.1]{B21} we then have
\begin{equation}\label{EL3.5B}
\psi_D(x)=\Exp_x\left[e^{\lambda_D t} \psi_D(X_t)\Ind_{\{\uptau>t\}}\right],
\quad x\in D.
\end{equation}
Using \eqref{EL3.5B}, \cref{L3.4} and the proof of \cite[Corollary~4.1]{BL17b} we get \eqref{EL3.5A}. This completes the proof.
\end{proof}

Now we are ready to prove the Faber-Krahn inequality.
\begin{theorem}[Faber-Krahn inequality]\label{T3.2}
Let
$z\mapsto j(z)$  be isotropic and radially decreasing.
Let $D$ be any bounded domain satisfying $|\partial D|=0$.
Then
$$\lambda_D\geq \lambda_B,$$
where $B$ is ball around $0$ satisfying $|B|=|D|$.
\end{theorem}

\begin{proof}
By the assumption on $j$, \hyperlink{A1}{(A1)} holds.
We note that $p_t(x, y)=p_t(y-x)$ where
$$\int_{\Rd} e^{i \xi\cdot x} p_t(x)\D{x}= e^{-t (|\xi|^2 + \psi(\xi))}.$$
From \cite{W83} we know that $p_t$ is isotropic unimodal, that is, $p_t$ is 
radially decreasing.
We first assume that $D$ is a smooth domain.
Without any loss of generality we may also assume $0\in D$. Then by Markov property
\begin{align*}
\Prob_0(\uptau>t)&=\lim_{m\to\infty} \Prob_{0}(X_{\frac{t}{m}}\in D,
X_{\frac{2t}{m}}\in D, \ldots, X_{\frac{mt}{m}}\in D)
\\
&= \lim_{m\to\infty}\, \int_{D}\int_D\cdots\int_D
p_{\frac{t}{m}}(x_1) p_{\frac{t}{m}}(x_2-x_1)\cdots p_{\frac{t}{m}}(x_m-x_{m-1})
\D{x_1}\D{x_2}\ldots\D{x_m}
\\
&\leq \lim_{m\to\infty}\, \int_{B}\int_B\cdots\int_B
p_{\frac{t}{m}}(x_1) p_{\frac{t}{m}}(x_2-x_1)\cdots p_{\frac{t}{m}}(x_m-x_{m-1})
\D{x_1}\D{x_2}\ldots\D{x_m}
\\
&=\Prob_{0}(\uptau_B>t),
\end{align*}
where in the third line we used Brascamp-Lieb-Luttinger inequality
\cite[Theorem~3.4]{BLL}. Therefore,
\begin{equation*}
-\lim_{t\to\infty} \frac{1}{t}\log\Prob_0(\uptau>t)
\geq -\lim_{t\to\infty} \frac{1}{t}\log\Prob_0(\uptau_B>t).
\end{equation*}
Applying \cref{L3.5} we then have
\begin{equation}\label{ET3.2A}
\lambda_D\geq \lambda_B\,.
\end{equation}
Now given a bounded domain $D$ with $|\partial D|=0$ we consider
 a decreasing sequence of smooth domains $D_n$
such that $\cap_{n\geq 1} D_n =\bar{D}$ and $|D_n|\to |D|$ as $n\to\infty$.
Let $B_n$ be a ball centered at $0$ and $|B_n|=|D_n|$. It is also easily seen that $B$ and $\{B_n\}$ satisfies condition \hyperlink{A4}{(A4)}. 
Using \eqref{ET3.2A} and monotonicity of $\lambda_D$ with respect to
domains, we get that
$$\lambda_D\geq \lambda_{D_n}\geq \lambda_{B_n}.$$
Now let $n\to\infty$ and apply \cref{L3.4} to conclude
$$\lambda_D\geq \lambda_B.$$
Hence the proof.
\end{proof}

\section{Symmetry of positive solutions and Gibbons' problem}\label{S-Gib}
In this section we prove radial symmetry of positive solutions and
a one-dimensional symmetry result related to the Gibbons' conjecture.

\begin{theorem}\label{T4.1}
Assume \hyperlink{A1}{(A1)}. Also, assume that $j$ is 
radially decreasing in $\Rd\setminus\{0\}$ and strictly decreasing in a 
neighbourhood of $0$.
Suppose that $D$ is convex in the direction of the $x_1$ axis, and symmetric about the plane $ \left\lbrace x_1 = 0 \right\rbrace$.  Also, let $f : [0,\infty) \rightarrow \RR$ be locally Lipschitz continuous. Consider any solution $u\in\cC_b(\Rd)$ of
\begin{align*}
L u &= f(u) \quad \text{in}\; D, 
\\
u &= 0 \quad \text{in}\; D^c,
 \\
u &>0 \quad \text{in}\; D .
\end{align*}
Then $u$ is symmetric with respect to $x_1 = 0$ and strictly decreasing in the $x_1$ direction.
\end{theorem}

\begin{proof}
We use the method of moving plane appeared in the seminal work
Gidas, Ni and Nirenberg \cite{GNN} which was motivated by 
a work of Serrin \cite{Serrin}. The proof
can be easily completed following the arguments of \cite{FW14}. We provide
a proof for the sake of completeness. 
Define
\begin{align*}
\Sigma_{\lambda} = \left\lbrace x=(x_1,x^{'} ) \in D : x_1>\lambda \right\rbrace \quad &and \quad T_{\lambda} =  \left\lbrace x=(x_1,x^{'} ) \in \Rd : x_1 =\lambda \right\rbrace, \\
u_{\lambda}(x) = u(x_{\lambda}) \quad &and \quad w_{\lambda }(x) = u_{\lambda}(x) - u(x),
\end{align*}
where $x_{\lambda} = (2\lambda - x_1, x^{'})$.  For a set $A$ we denote by $\sR_{\lambda}A$ the reflection of $A $ with respect to the plane $T_{\lambda}$.  Also, define 
$$\lambda_{\max}= \sup \left\lbrace \lambda > 0 : \Sigma_{\lambda} \neq \emptyset \right\rbrace .$$
We note that for any $\lambda \in (0, \lambda_{ \max})$, $u_{\lambda}$ is a viscosity solution of
$$L u_{\lambda} = f(u_{\lambda}) \quad \text{in}\; \Sigma_{\lambda} ,$$
and therefore, from \cref{T1.2} we obtain
$$
L w_{\lambda}= f(u_{\lambda}) - f(u) \quad \text{in}\; \Sigma_{\lambda}.
$$
Define $\Sigma_{\lambda}^{-} = \left\lbrace x \in \Sigma_{\lambda} : w_{\lambda} <0\right\rbrace$.   Since $w_{\lambda} \geq 0$ on $\partial \Sigma_{\lambda}$,  it follows that $w_{\lambda} = 0$ on $\partial \Sigma_{\lambda}^{-}$.  Hence the function
$$
v_{\lambda} = 
\begin{cases}
w_{\lambda} & \text{ in} \; \Sigma_{\lambda}^{-},
\\
0 & \text{elsewhere},
\end{cases}
$$
is in $\cC_b(\Rd)$.  We claim that for every $\lambda \in (0, \lambda_{ \max})$
\begin{equation}\label{ET4.1A}
L v_{\lambda} \leq f(u_{\lambda}) - f(u) \quad in \quad \Sigma_{\lambda}^{-},
\end{equation}
in the viscosity sense.
To see this, let $\varphi$ be a test function that touches $v_{\lambda}$ from below at a point $x \in \Sigma_{\lambda}^{-}$. Then  
$\varphi + (w_{\lambda	} - v_{\lambda}) \in \cC_b(x)$ and touches $w_{\lambda}$ at $x $ from below.  Denote $\zeta_{\lambda} (x) = w_{\lambda} - v_{\lambda}$.  It then follows that
$$L(\phi + \zeta_{\lambda}) (x) \leq f(u_{\lambda}(x)) - f(u(x)). $$
Since $\zeta_{\lambda}  = 0$ in a small neighbourhood of $x$ in 
$\Sigma_{\lambda}^{-}$ we have $\Delta \zeta_{\lambda}(x)=0$. 
Again, since $j$ is radial, to show \eqref{ET4.1A} it is enough to
show that
$$ \int_{\Rd} ( \zeta_{\lambda}(x+z)-\zeta_\lambda(x)) j(z) \D{z} \geq 0.
 $$
This can be done by following the argument of \cite[p.~8]{FW14} and 
the fact $j$ is radially decreasing.

If $\lambda < \lambda_{\max}$ is sufficiently close to $\lambda_{\max}$, then $w_{\lambda} > 0$ in $\Sigma_{\lambda}$. Indeed, note that if $\Sigma_{\lambda}^{-} \neq \emptyset$,
then $v_{\lambda}$ satisfies \eqref{ET4.1A}. Denoting
$$c(x) = \frac{f(u_{\lambda}(x)) - f(u(x))}{u_{\lambda}(x) - u(x)},$$
 it then follows that
$$
L v_{\lambda} - c(x) v_{\lambda} \leq 0 \quad \text{in} \quad \Sigma_{\lambda}^{-} .
$$ 
 Thus choosing $\lambda$ sufficiently close to $\lambda_{\max}$  it follows from \cref{C2.1} that $v_{\lambda} \geq 0$ in $\Rd$.  Hence $\Sigma_{\lambda}^{-} = \emptyset$ and we have a contradiction. To show that $w_{\lambda} > 0$ in $\Sigma_{\lambda}$,  we assume to the contrary that $w_{\lambda}(x_0) = 0$ for some $x_0 \in \Sigma_{\lambda}$.
Consider a non-negative test function $\varphi \in \cC_b(x_0)$ crossing $w_{\lambda}$ from below with the property that $\varphi = 0$ in $B_r(x_0) \Subset \Sigma_{\lambda}$ and $\varphi = w_{\lambda}$ in $B^c_{2r}(x_0)$. Furthermore, choose $r$ small enough such that $B_{2r}(x_0) \Subset \Sigma_{\lambda}$ and $\varphi \geq 0$ in $ \Sigma_{\lambda}$. Then we obtain 
\begin{equation}\label{ET4.1B}
L\varphi(x_0) \leq 0 .
\end{equation}
Since $\Delta \varphi(x_0) = 0$ we get 
$$I[\varphi](x_0):= \frac{1}{2}\int_{\Rd}(\varphi(x_0+z)+\varphi(x_0-z)-2\varphi(x_0))
j(z)\D{z} \leq 0.$$ 
Next we compute $I[\varphi](x_0)$.
Note that $\varphi \geq 0$ in $R_{\lambda} \df \left\lbrace x \in \Rd \; : \; x_1 \geq \lambda  \right\rbrace$. We have
 \begin{align*}
I[\varphi](x_0) &= \int_{\Rd} \varphi(z) j(\vert x_0 - z \vert) \D{z}
\\
&= \int_{R_{\lambda}} \varphi(z) j(\vert x_0 - z \vert) \D{z} + \int_{\sR_{\lambda} R_{\lambda}} w_{\lambda}(z) j(\vert x_0 - z \vert) \D{z}
\\	
&= \int_{R_{\lambda}} \varphi(z) j(\vert x_0 - z \vert) \D{z} + \int_{ R_{\lambda}} w_{\lambda}(z_{\lambda}) j(\vert x_0 - z_{\lambda} \vert) \D{z}
\\	
&= \int_{R_{\lambda}} \varphi(z) j(\vert x_0 - z \vert) \D{z} - \int_{ R_{\lambda}} w_{\lambda}(z) j(\vert x_0 - z_{\lambda} \vert) \D{z}
\\	
&\geq \int_{R_{\lambda} \setminus B_{2r}(x_0)} w_{\lambda}(z) \left( j(\vert x_0 - z \vert) - j(\vert x_0 - z_{\lambda} \vert) \right) \D{z} - \int_{  B_{2r}(x_0)} w_{\lambda}(z) j(\vert x_0 - z_{\lambda} \vert) \D{z}.
 \end{align*}
Since $\vert z_{\lambda} - x_0\vert > \vert z - x_0\vert$ and $j(\vert z_{\lambda} - x_0\vert) \geq j(\vert z - x_0\vert)$ thus the first term in the above expression is non-negative. In fact, since
$$
\lim_{r \rightarrow 0} \int_{  B_{2r}(x_0)} w_{\lambda}(z) j(\vert x_0 - z_{\lambda} \vert) \D{z} =0,
$$
using \eqref{ET4.1B}, we obtain
\begin{equation}\label{ET4.1C}
\lim_{r \rightarrow 0} \int_{R_{\lambda} \setminus B_{2r}(x_0)} w_{\lambda}(z) \left( j(\vert x_0 - z \vert) - j(\vert x_0 - z_{\lambda} \vert) \right) \D{z}\leq 0.
\end{equation}
Since $w_\lambda$ is continuous and $j$ is strictly decreasing in a neighbourhood of $0$, from \eqref{ET4.1C} we get $w_\lambda=0$ in
$B_\delta(x_0)$ for some $\delta>0$. Thus $\{w_\lambda=0\}\cap\Sigma_\lambda$ is an open set. Hence $\{w_\lambda=0\}$ forms
a connected component of $\Sigma_\lambda$ which in turn, implies that
$\{w_\lambda=0\}\cap\partial\Sigma_\lambda\cap\partial D\neq\emptyset$.
 This is a contradiction.
Hence we must have $w_{\lambda} > 0$ in $\Sigma_{\lambda}$.

Now from the above argument and Step 2 in \cite[p.~10]{FW14} we can show that $\inf \{ \lambda > 0 \;  :\; w_{\lambda} > 0 \quad \text{in}\;\Sigma_{\lambda} \} = 0$. Also, strict monotonicity of $u$ in
the $x_1$ direction can be obtained by following the calculations in Step 3 of \cite{FW14}.
\end{proof}

As a consequence of \cref{T4.1} we obtain.
\begin{corollary}\label{C4.1}
Grant the setting of \cref{T4.1}.
 Then every solution to
\begin{align*}
L u= f(u) \quad \text{in}\; B_1(0), \quad
u=0\quad \text{in}\; B^c_1(0), \quad \text{and}\quad u>0 
\quad \text{in}\; B_1(0),
\end{align*}
is radial and strictly decreasing in $|x|$.
\end{corollary}
Remaining part of this section is devoted to the Gibbons' problem.
Let $u : \Rd \rightarrow \RR$ be a solution to the problem
\begin{equation}\label{E4.4}
 \begin{cases}
L u(x) = f(u(x)), & \quad \text{for} \; x \in \Rd,
\\[2mm]
\lim_{x_n \rightarrow \pm \infty } u(x^{'}, x_n) = \pm 1 , & \quad \text{uniformly for } \;  x^{\prime} \in \RR^{d-1} .
\end{cases}
\end{equation}
We also suppose that $f \in \cC^1(\RR)$ satisfying
\begin{equation}\label{E4.5}
\inf_{|r|\geq 1}  f^{'}(r) >0\,.
\end{equation}
We show that $u$ is one-dimensional.
\begin{theorem}\label{T4.2}
Assume \hyperlink{A1}{(A1)}. Let $u\in\cC_b(\Rd)$ solve \eqref{E4.4} where $f$ satisfies \eqref{E4.5}.
Then there exists a strictly increasing function $u_0:\RR\to\RR$ satisfying
$$u(y, t)=u_0(t)\quad \text{for all}\; y\in\RR^d, \; t\in\RR.$$
\end{theorem}

We need the following lemma to prove \cref{T4.2}.
\begin{lemma}\label{L4.1}
Let $w \in\cC_b(\Rd)$ satisfy
$$ Lw - c(x)w = 0 \; \text{ in} \; \Rd ,$$
with
$$ w(x) \geq 0 \quad \text{in}\; \Rd \setminus U \;  \text{and} \; c(x) \geq \kappa \; \text{in} \; U$$
for some open set $U \subseteq \Rd$ and some constant $ \kappa> 0$.
Also, assume that $c\in\cC_b(\Rd)$. Then
$$ w(x) \geq 0 \; \text{ for all} \; x \in \Rd .$$
\end{lemma}

\begin{proof}
Suppose, to the contrary, that $m = \inf_{\Rd} w <0$. If $m$ is attend then
the proof can be completed from maximum principle.
In general, without loss of generality,  we may choose a sequence $x_k\in \Rd$ satisfying
\begin{equation}\label{EL4.10}
\lim_{k \to \infty} w(x_k) = m \quad \text{and} \quad w(x_k) \leq \frac{m}{2} < 0 \quad \forall\; k \in \NN .
\end{equation}
By given condition, for every $k \in \NN$, we have
$$
x_k \in U \quad \text{and} \quad c(x_k) \geq \kappa >0 .
$$
By \cite[Theorem~4.1]{Mou19} there exist $\hat\kappa, \alpha>0$, dependent on $j, d,\norm{c}_{L^\infty (\Rd)}$, such that 
\begin{equation}\label{EL4.1A}
\sup_{x\neq y: x, y\in B_{\nicefrac{1}{2}}(0)} \frac{|w(x)-w(y)|}{|x-y|^{\alpha}}
\leq \hat\kappa \, \norm{w}_{L^\infty (\Rd)}.
\end{equation}
Translating the center of the ball it is evident from \eqref{EL4.1A} that
$w\in\cC^{\alpha}_b(\Rd)$. We note that for some $\delta>0$ we have
$\dist(x_k, U^c)>\delta$. Otherwise, along some subsequence, we must have
$|x_k-z_k|\to 0$ as $k\to \infty$, for some $z_k\in U^c$. Since 
$w\in\cC^\alpha_b(\Rd)$ and $w\geq 0 $ in $U^c$, we get
$$w(x_k)\geq w(z_k)-\hat\kappa |x_k-z_k|^\alpha\to 0, \quad \text{as}\;k\to\infty\,.$$
This is a contradiction to \eqref{EL4.10}. Thus, we
must find $\delta>0$ so that $B_\delta(x_k)\Subset U$.

Let us now define $v_k(x)=w(x_k+x)$. 
Using \eqref{EL4.10} and \eqref{EL4.1A}, we restrict $\delta$ small enough
so that 
$$w(y)<\frac{m}{4}\quad \text{in}\; B_\delta(x_k), \quad \text{for all}\; k.$$
Thus, by the given condition on $c$, it follows that
\begin{equation}\label{EL4.1B}
L v_k\leq \kappa\frac{m}{4}\quad \text{in}\; B_\delta(0).
\end{equation}
Since $\{v_k\}$ forms a equicontinuous
family, using Arzel\`a-Ascoli theorem, we can find a $v\in\cC_b(\Rd)$
satisfying $v_k\to v$ along some subsequence, uniformly over compacts.
Using the stability property of viscosity supersolutions, it follows from
\eqref{EL4.1B} that
\begin{equation}\label{EL4.1C}
Lv\leq \kappa\frac{m}{4}\quad \text{in}\; B_\delta(0).
\end{equation}
On the other hand
$$
v(0) = \lim_{k \to \infty} v_k(0) = \lim_{k \to \infty} w(x_k) = m  \leq  
\lim_{k\to\infty}w(x+x_k) = v(x), 
$$
for all $x \in \Rd$.
Thus $x = 0$ is a minimum point for $v$ in $\Rd$.  Then $\varphi\equiv m$
is a bonafide test function at $x=0$. From \eqref{EL4.1C} we then obtain
$$0>\frac{m}{4}\kappa\geq L\varphi(0)\geq 0,$$
which is a contradiction. Therefore, we must have $m\geq 0$ which completes
the proof.
\end{proof}

Now we can complete the proof of \cref{T4.2}.
\begin{proof}[Proof of \cref{T4.2}]
We broadly follow the idea of \cite{FV11,BVDV} without imposing any stronger
regularity assumption on $u$.
Fix a unit vector $\nu=(\nu_1, \ldots, \nu_d)$ such that $\nu_d > 0$ 
and we write $\nu = (\nu^{\prime}, \nu_d)$.  We also define
$$
\Gamma_h[u](x) = u(x+ h \nu) - u(x).
$$
We first show that $\Gamma_h[u](x) >0$ for all $x \in \Rd$ and for all $h >0$. Observe that 
$$
f(u(x+h \nu)) - f(u(x)) = c_h(x) \Gamma_h[u](x),
$$
where 
$$c_h(x) = \int_0^1 f^{\prime}\left( (1-t) u(x) + tu(x+h\nu)   \right) \D{t}. 
$$
Using \eqref{E4.5} we can choose $\delta \in (0 , \frac{1}{2})$ such that $f^{\prime} \geq \kappa_1$ in $(- \infty, -1+\delta ] \cup [ 1- \delta , \infty)$ for some $\kappa_1>0$. Again, by \eqref{E4.4},
we may take $M>0$ satisfying
\begin{equation}\label{ET4.2A}
u(x) \geq  1- \delta \quad \text{for} \quad x_d \geq M, \quad \text{and} \quad u(x) \leq -1+ \delta \quad \text{for} \quad x_d \leq -M .
\end{equation}
{\bf Claim}: If $x \in \{\Gamma_h[u] < 0\}  \cap \{|x_d| \geq M\}$, then $c_h(x) \geq \kappa_1$. 

If $x \in \{\Gamma_h[u] < 0\}\cap\{ x_d  \leq -M\}$,
then $u(x + h \nu) < u(x) \leq -1 + \delta$, by \eqref{ET4.2A}, which implies
$$
(1-t) u(x) + t u(x + h \nu)  \leq - 1+ \delta ,
$$
for all $t \in [0,1]$. Similarly, if $x \in\{\Gamma_h (u) < 0\}\cap \{x_d  \geq M\}$, then $ u(x) > u(x + h \nu) \geq 1 - \delta$, since $x_d + h \nu_d > x_d \geq M$. Thus
$$
(1-t) u(x) + t u(x + h \nu)  \geq  1- \delta ,
$$
for all $t \in [0,1]$.  Hence we have $c_h(x) \geq \kappa_1$ for  $x\in\{\Gamma_h[u] < 0 \}\cap \{|x_d| \geq M\}$.

Next we claim that if $h \geq \frac{2M}{\nu_d}$, then $\Gamma_{h}[u](x) > 0$ for any $x \in \Rd$. 
Fix any $h\geq \frac{2M}{\nu_d}$ and let
$U = \{ \Gamma_h (u) < 0 \}$. 
For $x_d = -M$ we have $x_d+h \nu_d \geq M$ and therefore,
$$
\Gamma_h[u](x) \geq \inf_{x_d \geq M} u(x) - \sup_{x_d \leq -M} u(x) 
\geq 1-\delta- (-1+\delta)= 2(1-\delta) >0.
$$
Thus $U\subset\{x_d=-M\}^c$. We write $U=U^{-}\cup U^{+}$ where
$U^{-}=U\cap\{x_d < -M\}$ and $U^{+}=U\cap\{x_d > -M\}$.
By above claim, $c_h(x) \geq \kappa_1$ for $x \in U^{-}$.  
Again, for $x \in U^{+}$ we have $x_d + hv_d > M$ by our choice of $h$
and hence, we have $u(x) > u(x + h\nu) \geq 1 - \delta$. Thus
$$
(1-t) u(x) + t u(x + h \nu)  \geq  1- \delta ,
$$
for all $t \in [0,1]$.  Hence we have $c_h(x) \geq \kappa_1$ 
for all $x \in U$ and $\Gamma_h[u] \geq 0$ in $U^c$.  Applying
\cref{T1.2} we also have 
$$L \Gamma_h[u] = c_h\, \Gamma_h[u]\quad  \text{in}\; \Rd.$$
By \cref{L4.1} we then obtain $\Gamma_h[u] \geq 0$ in $\Rd$.  
We can apply
strong maximum principle \cref{T2.3} to get $\Gamma_h (u)(x) > 0$ in $\Rd$,  since $\Gamma_h[u]> 0$ on $\{x_d = -M\}$. This proves
the claim that $\Gamma_h[u] > 0$ in $\Rd$ and for 
$h \geq \frac{2M}{\nu_d}$. Define
\begin{equation}\label{ET4.2A0}
h_\circ = \inf\{ h>0 \;: \; \Gamma_s[u](x) > 0 \; \text{for all} \; x \in \Rd \; \text{with} \; \vert x_d \vert \leq M, \; \text{for all}\; s\geq h \} .
\end{equation}
From the above argument we have $h_\circ\in[0, \frac{2M}{\nu_d}]$. 
We show that $h_\circ = 0$. Suppose, to the contrary, that $h_\circ>0$.
Then for any $\varepsilon \in (0, h_\circ)$
$$
\Gamma_{h_\circ+\varepsilon}[u](x)=u(x + (h_\circ + \varepsilon) \nu) -u(x) > 0 \; \text{for all} \; x \in \{ \vert x_d \vert \leq M \},$$
and
$$
\Gamma_{h_\circ-\varepsilon_k}[u](x_k)=u(x_k + (h_\circ - \varepsilon_k)\nu) - u(x_k) \leq 0 \; \text{for some} \; x_k \in \{ \abs{x_d} \leq M \} ,
$$
for any sequence $\varepsilon_k\to 0$ as $k\to\infty$.  Since
$u$ is continuous, we obtain
\begin{equation}\label{ET4.2B}
\Gamma_{h_\circ}[u](x) = \lim_{\varepsilon \to 0}u(x +(h_\circ + \varepsilon) \nu) - u(x) \geq 0,
\end{equation}
for any $x \in \{ \vert x_d \vert \leq M \}$.  Repeating an
argument similar to above would give $\Gamma_{h_0}[u] \geq 0 $ in $\Rd$.
Now we define $w_k = u(x+x_k)$. Since $u\in\cC^\alpha_b(\Rd)$ by \eqref{EL4.1A}, $\{w_k\}$ is equicontinuous and therefore,
 $w_k \to w_{\infty}$ along some subsequence,  uniformly on compacts.
As a consequence, 
\begin{align*}
c_k(x):=c_{h_\circ-\varepsilon_k}(x+x_k)
&=\int_0^1 f^{\prime}\left( (1-t) w_k(x) + t w_k(x+(h_\circ-\varepsilon_k)\nu)\right)\D{t}
\\
&\to \int_0^1 f^{\prime}\left( (1-t) w_\infty(x) + t w_\infty(x+h_\circ \nu)\right)\D{t}:=c_\infty(x)
\end{align*}
uniformly over compacts. From the stability property of viscosity solution
we then obtain
$$
L\, \Gamma_{h_\circ}[w_{\infty}](x) = f(w_{\infty}(x + h_\circ \nu)(x)) - f(w_{\infty}(x)) = c_{\infty}(x)\Gamma_{h_\circ}[w_{\infty}](x)
\quad \text{in}\; \Rd.
$$
Again, by \eqref{ET4.2B}
$$\Gamma_{h_\circ}[w_\infty](x)
= \lim_{k\to\infty} \Gamma_{h_\circ}[u](x+x_k)\geq 0.$$
Also, from \eqref{EL4.1A} we have
\begin{align*}
\Gamma_{h_\circ}[w_{\infty}](0) &= \lim_{k \to \infty} u(x_k + h_\circ \nu) - u(x_k)
\\
&\leq  \lim_{k \to \infty} u(x_k + (h_\circ - \varepsilon_k)  \nu) - u(x_k) + (\varepsilon_k)^{\alpha} \norm{u}_{\cC^{\alpha}} 
\\
&\leq 0 .
\end{align*}
Hence $\Gamma_{h_\circ} w_{\infty}(0) = 0$. By strong maximum principle (\cref{T2.3}) we must have $\Gamma_{h_\circ} w_{\infty}(x) = 0$ for all $x \in \Rd$.
By a simple iteration this also gives
$$
w_{\infty}(x + jh_\circ \nu) = w_{\infty}(x) 
$$
for any $x$ and $j \in \ZZ$.  Choosing $j \in \NN \cap [ \frac{2M}{h_\circ \nu_d} , \infty)$ we see that $j h_\circ \nu_d + (x_d)_k \geq M$ and $-j h_\circ \nu_d + (x_d)_k \leq -M$ (since $x_k\in\{|x_d|\leq M\}$)
and therefore, by \eqref{ET4.2A} we obtain $u( x_k +jh_\circ \nu) \geq 1-\delta $ and $u(x_k - jh_\circ\nu ) \leq -1 + \delta$ for all $k$. 
Hence
\begin{align*}
2(1-\delta) &\leq \lim_{k \to \infty} u( x_k +jh_\circ \nu)  - u(x_k - j h_\circ\nu ) 
\\
&= w_{\infty} (jh_\circ \nu) - w_{\infty} (- jh_\circ \nu)=0,
\end{align*}
which is a contradiction. Thus $h_\circ$ in \eqref{ET4.2A0} must be
$0$. In other words,
for any $h> 0$,  $\Gamma_h[u]> 0$ in $\{ \vert x_d \vert \leq M \}$.
Since $c_h(x) > \kappa_1 $ for any $x \in \{ \Gamma_h [u] < 0\} \cap \{\vert x_d \vert \geq M \} $,  by the same argument as above 
we have $\Gamma_h[u] > 0$ in  $\Rd$ .

Thus we have proved that $ \Gamma_{h}^{\nu}[u](x)\df u(x+ h \nu) - u(x) \geq 0$ for all  $\nu \in S^{d-1}$ with $\nu_d > 0$ and all $h \geq 0$. Taking $\mu = - \nu$ we obtain for all $h\geq 0$ that
$$
\Gamma_h^{\mu}[u] \leq 0\quad \text{for all}\; x \in \Rd, \; \text{and} \; \text{for all}\; \mu \in \mathbb{S}^{d-1} \; \text{with} \; \mu_d < 0,
$$
as
\begin{align*}
\Gamma_h^{\mu}[u](x) &= u(x+ h \mu) - u(x) = 
u(\tilde{x}) -u(\tilde{x}+ h (-\mu))
\\
&= - (u(\tilde{x}+ h \nu) - u(\tilde{x})) = - \Gamma_h^{\nu}[u](\tilde{x})
\leq 0\, ,
\end{align*}
where $\tilde{x} = x + h \mu$. Now letting $\mu_d\nearrow 0$ and
$\nu_d\searrow 0$ it follows from above that
$$
\Gamma_h^{\omega}[u] = 0 \;\; \text{for all}\; x \in \Rd, \;\; \text{and} \; \;\text{for all}\; \omega \in \mathbb{S}^{d-1} \; \text{with} \; \omega_d = 0.
$$ 
In particular, this gives $\partial_{x_i} u=0$ for $i=1, 2,\ldots, d-1$.
Again, $u_0$ is strictly increasing follows from \eqref{ET4.2A0} and the fact that $h_\circ=0$.
This completes the proof of the theorem.

\end{proof}

\section{Existence and uniqueness results}\label{S-EU}
Our goal in this section is to prove the existence of a unique viscosity 
solution to
\begin{equation}\label{E5.1}
Lu = -f \quad \text{in}\; D, \quad \text{and}\quad 
u=g \quad \text{in}\; D^c,
\end{equation}
where $f\in \cC(\bar{D})$ and $g\in \cC(D^c)$.
Denote by $\cC^2_b(x)$ the space of all bounded  functions
in $\Rd$ that are twice continuously differentiable in some neighbourhood around $x$. Also, recall the definition
of viscosity solutions from \cref{Defi1.1}.
%

First we prove \cref{T1.2}.
To do so, we need the notion inf and sup convolution.
Given a bounded, upper-semicontinuous function $u$, the sup-convolution
approximation $u^{\varepsilon}$ is given by
$$ u^{\varepsilon} (x) 
= \sup_{y \in \Rd} u(x+y) - \frac{\abs y^2}{\varepsilon} 
= \sup_{y \in \Rd} u(y) - \frac{\vert x- y\vert^2}{\varepsilon}
= u(x^*) - \frac{\abs{x-x^*}^2}{\varepsilon}.$$
Likewise, for a bounded and lower-semicontinuous function $u$, the inf-convolution $u_{\varepsilon}$ is given by
$$ u_{\varepsilon} = \inf_{y \in \Rd} u(x+y) + \frac{\abs y^2}{\varepsilon} =   \inf_{y \in \Rd} u(y) + \frac{\vert x- y\vert^2}{\varepsilon}.$$

\begin{lemma}\label{L5.1}
Let $D$ be an open bounded set and $f$ is continuous function in $D$.  If $u$ is a bounded, upper-semicontinuous function such that $L u \geq f$ in $D$, then $L u^{\varepsilon} \geq f - d_{\varepsilon}$ in $D_1 \Subset D$ where $d_{\varepsilon} \rightarrow 0$ in $D_1$, as $\varepsilon \rightarrow 0 $, and depends on the modulus of continuity of $f$.

An analogous statement also holds for supersolutions.
\end{lemma}

\begin{proof}
Fix $x_0 \in D_1 $, and let $\varphi$ be a test function that touches $u^{\varepsilon}$ from above at $x_0$ in some neighbourhood  $B_r (x_0) \subset D_1$  and $\varphi = u^{\varepsilon}$ in $B_r^c(x_0)$.  We define
$$
Q(x) = \varphi(x+x_0 - x_0^{\ast}) + \frac{1}{\varepsilon} \vert x_0 - x_0^{\ast} \vert^2.
$$
 We observe from the definition of $u^\varepsilon$ that $|x_0 - x_0^{\ast} | \leq M$  where $M = (2 \Vert u \Vert_{L^{\infty}})^{1/2}$. 
Hence we can pick $\varepsilon_1 $ such that for all $\varepsilon \leq \varepsilon_1$ and $x_0 \in D_1$ we have $x_0^{\ast} \in D$.
It then follows from the definition that $u(x) \leq u^{\varepsilon } (x+x_0 - x_0^{\ast}) + \frac{1}{\varepsilon} \vert x_0 - x_0^{\ast} \vert^2$.
Thus, for $\vert x - x_0^{\ast} \vert < r$ we then get
$$
u(x) \leq \varphi(x+x_0 - x_0^{\ast}) + \frac{1}{\varepsilon} \vert x_0 - x_0^{\ast} \vert^2 = Q(x)
$$
and $u(x_0^{\ast}) = Q(x_0^{\ast})$. Hence Q touches $u$ by above at $x_0^{\ast}$ in $B_r(x_0^{\ast})$. Define
$$
w(x) = \begin{cases}
Q(x) & x \in B_r(x_0^{\ast}), \\[2mm]
u(x) & x \in \Rd \setminus B_r(x_0^{\ast}).
\end{cases}
$$
Then by the definition of viscosity subsolution we have
$Lw(x_0^{\ast})  \geq f(x_0^{\ast})$, that is,
\begin{align*}
\Delta Q(x^*_0) + \int_{\Rd} (w(x_0^{\ast} + y) - Q(x_0^{\ast})-\Ind_{\{|y|\leq 1\}} y\cdot\grad Q(x_0^*) ) j(y) \D{y}  \geq f(x_0^{\ast}).
\end{align*}
Now we observe that $\Delta Q(x_0^{\ast}) = \Delta\varphi(x_0)$, $\grad Q(x^*_0)=\grad\varphi(x_0)$ and 
$Q(x^*_0+y)-Q(x^*_0)=\varphi(x_0+y)-\varphi(x_0)$.
Since $\varphi\geq u^\varepsilon\geq u$, we obtain
$$
L\varphi(x_0)  \geq f(x_0) - \vert f(x_0^{\ast}) - f(x_0) \vert.
$$
Since $\vert x_0 - x_0^{\ast} \vert \leq M  \varepsilon^{\frac{1}{2}}$, 
 choosing $d_{\varepsilon}(x) = \sup_{y \in B_{M\sqrt{\varepsilon}} (x_0)} \vert f(x_0^{\ast}) - f(x_0) \vert$ gives us the desired result.
\end{proof}

Next we show that difference of sub and supersolution gives us a 
subsolution.
\begin{lemma}\label{L5.2}
Let $D$ is open bounded set,  $u$ and $v$ be two bounded functions such that u is upper-semicontinuous and v is lower-semicontinuous in $\Rd$, $L u \geq f$ and $L v \leq g$   in $D$ for two continuous function $f$ and $g$ then $L (u-v) \geq f-g$ in $D$.
\end{lemma}

\begin{proof}
Fix $D_1 \Subset D$ and $\varepsilon >0 $. Let 
$P\in\cC^2_b(x_0)$ be such that $u^{\varepsilon} - v_{\varepsilon} \leq P$ in $B_r(x_0) \subset D_1$ and $u^{\varepsilon} (x_0)- v_{\varepsilon}(x_0) = P(x_0)$. Without loss of generality  we may also assume that
$P$ is a paraboloid and $B_{2r} (x_0) \subset D$.  Take $\delta > 0$ and define
$$
w(x) = v_{\varepsilon}(x) - u^{\varepsilon}(x) + \phi(x) + \delta (\vert x - x_0 \vert\wedge r)^2-  \delta r_1^2 ,
$$
where $0 < r_1 < \delta\wedge \frac{r}{2}$ and
$$
\phi(x) = \begin{cases}
P(x) & x \in B_r(x_0) \\[2mm]
u^{\varepsilon}(x) - v_{\varepsilon}(x) & x \in \Rd \setminus B_r(x_0).
\end{cases}
$$
We see that $w \geq 0$ on $\partial B_{r_1}(x_0)$,   $w(x_0) < 0$ and $w > \frac{3\delta}{4} r^2$ on $\Rd \setminus B_r(x_0)$.  For any $x \in \overline{B}_{r_1}(x_0)$ there exists a convex paraboloid $P^x$ on opening $K$ such that it touches $w$ from above at $x$ in $B_{r_1}(x)$, 
where $K (={4}/{\varepsilon})$ is a constant independent of $x$. Using \cite[Lemma~3.5]{CC}  and $w(x_0) < 0$ we obtain
\begin{equation}\label{EL1.2A}
0 < \int_{A \cap \left\lbrace w = \Gamma_w \right\rbrace } {\rm det} D^2\, \Gamma_w,
\end{equation}
where $\Gamma_w$ is the convex envelope of $w$ in $B_{2r}(x_0)$ and
 $A \subset B_{r_1}(x_0)$ satisfies $\vert B_{r_1}(x_0) \setminus A \vert = 0$, $\Gamma_w$ is second order differentiable 
 in $A$ and $\Gamma_w\in \cC^{1,1} (B_{\frac{r}{2}})$.
Furthermore, $u^{\varepsilon}$,  $v^{\varepsilon}$ (and hence $w$) are punctually second order differentiable in $A$ \cite[Theorem~5.1]{CC}.  Thus 
 $L u^{\varepsilon}(x)$,  $L v^{\varepsilon}(x)$ are defined in the classical sense for $x \in A$ and from \cref{L5.1} we have 
$$
L u^{\varepsilon}(x)\geq f(x) - d_{\varepsilon}  \quad \text{and } \quad L(v^{\varepsilon})(x) \leq g(x) + d_{\varepsilon}.
$$
We note that  since the contact set $\{w=\Gamma_w\}$ are the points of minimum for
$w-\Gamma_w$ and $w$ is differentiable at the points of $A$ (as it is punctually twice differentiable), we have $\grad w= \grad \Gamma_w$ on $A\cap \{w=\Gamma_w\}$.
Therefore, since $\Gamma_{w}$ is convex and $\Gamma_{w} \leq w$, we have
$$
\Delta w(x) \geq 0 \quad \text{ and } \quad 
\int_{x+y\in B_r(x)}(w(x+y)-w(x)-\Ind_{\{|y|\leq 1\}}y\cdot \grad w(x)) j(y)\D{y} \geq 0 \quad \text{for} \; x \in A \cap \left\lbrace w = \Gamma_{w} \right\rbrace,
$$
using the fact that for $x \in A \cap  \left\lbrace w = \Gamma_{w} \right\rbrace$ we have
$$w(x+y) - w(x)-\Ind_{\{|y|\leq 1\}} y\cdot \grad w(x) 
\geq \Gamma_{w}(x+y) -\Gamma_w(x) - \Ind_{\{|y|\leq 1\}} y\cdot \grad w(x) \geq 0 ,$$ for all $x+y\in B_r(x)$. 

Now from \eqref{EL1.2A} it follows that
$$
\vert \left\lbrace w = \Gamma_{w} \right\rbrace \cap A \vert > 0,
$$
and therefore, there is one point $x_1^{\delta} \in \left\lbrace w = \Gamma_{w} \right\rbrace \cap A$ where $L u^{\varepsilon}(x)$,  $L v^{\varepsilon}(x)$ can be computed classically.
At this point we thus have
\begin{align*}
f(x_1^{\delta}) - d_{\varepsilon} &\leq L u^{\varepsilon}(x_1^{\delta}) = 
Lv_{\varepsilon}(x_1^{\delta}) - Lw(x_1^{\delta}) + L\phi(x_1^{\delta})+
\delta L ( \vert  \bullet - x_0 \vert \wedge r)^2(x_1^{\delta}) 
\\
&\leq g(x_1^{\delta}) +d_{\varepsilon} +  L\phi(x_1^{\delta}) - \int_{|y| \geq r}  (w(x_1^\delta + y) -w(x_1^\delta)) j(y) \D{y}  \\
& \quad + 
\int_{r \leq |y|\leq 1} y\cdot \grad \Gamma_w(x^1_\delta) j(y)\D{y} +
\delta L ( \vert  \bullet - x_0 \vert \wedge r)^2(x_1^{\delta}).
\end{align*}
Letting $r_1\to 0$, we see that $x^\delta_1 \to x_0$ and 
$ \grad \Gamma_w(x_1^\delta)\to \grad \Gamma(x_0)$.  Since $\Gamma_w$ attains its minimum at $x_0$ we have $\grad\Gamma_w(x_0)=0$.  Also, $w(x_1^\delta+y) - w(x_1^\delta) \to w(x_0 + y ) - w(x_0)$. Hence, by the dominated convergence theorem, we have 
\begin{align*}
f(x_0) - d_{\varepsilon} &\leq g(x_0) +d_{\varepsilon} +  L\phi(x_0) - \int_{|y| \geq r}  (w(x_0 + y) -w(x_0)) j(y) \D{y}   +
\delta L ( \vert  \bullet - x_0 \vert \wedge r)^2(x_0).
\end{align*}
Since $w(x_0 + y ) - w(x_0) \geq 0$ for all $|y| \geq r$, we have
\begin{align*}
f(x_0) - d_{\varepsilon} &\leq g(x_0) +d_{\varepsilon} +  L\phi(x_0) + \delta L ( \vert  \bullet - x_0 \vert \wedge r)^2(x_0).
\end{align*}
Now, we let $\delta\to 0$ to
find that
 $$
 f(x_0) -g(x_0)  - 2 d_{\varepsilon}  \leq 
 L\phi(x_0). 
 $$
This gives us $L(u^\varepsilon-v_\varepsilon)\geq f-g -2d_\varepsilon$ in
$D_1$, in the viscosity sense. At the end, we let $\varepsilon\to 0$ and use the stability property of viscosity solution to obtain our desired result.
\end{proof}

Now applying a standard approximation argument together with \cref{L5.2}
 we obtain (cf. 
\cite[Theorem~5.9]{CS09})
\begin{theorem}\label{T5.1}
Let $D$ be an open bounded set,  $u$ and $v$ be two bounded functions such that $u$ is upper-semicontinuous and $v$ is lower-semicontinuous in $\overline{D}$ also $L u \geq f$ and $L v \leq g$ in the viscosity sense 
in $D$, for two continuous functions $f$ and $g$, then $L (u-v) \geq f-g$
in $D$ in the viscosity sense.
\end{theorem}

Now we can state our comparison result.

\begin{theorem}\label{T5.2}
Let u be a bounded function in $\Rd$ which
is upper-semicontinuous in $\overline{D}$ and satisfies $L u \geq 0$ in $D$. 
Then $\sup_{D} u \leq  \sup_{D^c} u$ .
\end{theorem}

\begin{proof}
From \cite[Lemma 5.5]{Mou19} we can find a non-negative function 
$\chi\in \cC^2(\bar{D})\cap\cC_b(\Rd)$ satisfying
$$L \chi\leq -1\quad \text{in}\; D.$$
The above equation holds in the classical sense.
For $ \varepsilon> 0$,
we let $\phi_M$ to be
$$
\phi_M (x) = M + \varepsilon \chi.
$$
Then $L\phi_M \leq - \varepsilon$ in $D$.

Let $M_0$ be the smallest value of $M$ for which $\phi_M \geq u$ in $\Rd$. We show that
$M_0 \leq \sup_{D^c } u$.  Suppose, to the contrary, that 
$M_0 >  \sup_{D^c } u$. Then there must be a point $x_0 \in D$ for which $u(x_0) = \phi_{M_0}(x_0)$.  But in that case $\phi_{M_0}$ would touch $u$ from above at
$x_0$ and thus, by the definition of the viscosity subsolution, we would have that  $L\phi_{M_0}(x_0) \geq 0$. This is a contradiction. Therefore,  
$M_0 \leq \sup_{D^c } u$ which implies that
for every $x \in \Rd$
$$
u \leq \phi_{M_0} \leq M_0 + \varepsilon\, \sup_{\Rd}\chi \leq \sup_{D^c} u +\varepsilon\, \sup_{\Rd}\chi.
$$
The result follows by taking $\varepsilon \rightarrow 0$.
\end{proof}
Now we prove the existence of viscosity solution to \eqref{E5.1}.
\begin{lemma}\label{L5.3}
Assume \hyperlink{A1}{(A1)} and let $D$ be a bounded Lipschitz domain.
Define
$$u(x)=\Exp_x\left[\int_0^{\uptau} f(X_t)\, \D{t}\right] + \Exp_x[g(X_\uptau)],\quad x\in D,$$
where $\uptau$ denotes the first exit time of $X$ from $D$.
Then $u\in \cC_b(\Rd)$ and solves \eqref{E5.1} in the viscosity sense.
\end{lemma}

\begin{proof}
Since $u(x)=g(x)$ in $D^c$, we only need to show that $u\in\cC(\bar{D})$.
Let $x_n \in D$ and $x_n \to z\in \bar{D}$.  Define
$$
\uptau_n = \inf \left\lbrace t > 0 \;:\;  X_t^n = x_n + X_t \notin D \right\rbrace.
$$
Here $\uptau_n$ is the first exit time of a process starting from $x_n$.  In a similar manner one can define the first exit time $\uptau_z$ of a process starting from $z$ as
$$
\uptau_z = \inf \left\lbrace t > 0\; :\;  X_t^z = z + X_t \notin D \right\rbrace.
$$
First suppose that $z\in \partial D$. Since $D$ is Lipschitz, it satisfies
the exterior cone condition and hence regular with respect to $X$ 
\cite{Millar, Sztonyk}. This means $P_z(\uptau_{\bar D}=0)=1$. Therefore,
for every $\delta>0$, $X^z$ intersects $(\bar{D})^c$ before time $\delta$,
almost surely. Since 
\begin{equation}\label{EL1.4A}
\sup_{t \in [0,M]} \vert x_n + X_t - (z + X_t) \vert \leq \vert x_n - z \vert \to 0, \quad \text{as} \quad n \to \infty,
\end{equation}
for every fixed $M$, it implies that $\uptau_n\to 0$
and $X^n_{\uptau_n}\to z$, almost surely. Therefore, using 
\cref{L2.1} and dominated convergence theorem, it follows that
$u(x_n)\to g(z)$ as $n\to\infty$.

For the remaining part we assume that $z\in D$.
For a fixed $M > 0$, next we show that $\uptau_n \wedge M \to \uptau_z \wedge M$ almost surely. Denote by $\Omega_M$ the event in \eqref{EL1.4A}.
Then $\Prob(\Omega_M)=1$. 
Since $D$ is regular we have $P(\uptau_z = \overline{\uptau}_z) = 1$, where $\overline{\uptau}_z$ is the first exit time of a process
from $\overline{D}$ starting at $z$  . Denote by $\tilde\Omega=\{ \uptau_z =  \overline{\uptau}_z \} $.  Let $\varepsilon > 0$.  We claim that,
on $\Omega_M\cap\tilde\Omega$, $\uptau_n \wedge M \leq \uptau_z \wedge M + \varepsilon$ for all large $n$. Also, we only need to show it on
$\{\uptau_z<M\}$.
For $\omega \in \Omega_M\cap\tilde\Omega \cap \left\lbrace \uptau_z < M \right\rbrace$, there exists $s \in [\uptau_z(\omega),  \uptau_z(\omega) + \frac{\varepsilon}{2}]$ such that $X^z_s(\omega) \in \overline{D}^c$
 which implies 
 $$\dist(X^z_s(\omega) , \overline{D})  >0.$$
By \eqref{EL1.4A}, it then implies that 
 $$
\uptau_n(\omega) \leq s \leq \uptau_z(\omega) + \frac{\epsilon}{2},
 $$
for all large $n$.  This proves the claim.

Next we show that, on $\tilde\Omega\cap\Omega_M $, we have  $\uptau_z \wedge M - \varepsilon \leq \uptau_n \wedge M$ for all large $n$. Now since $\uptau_z(\omega) \wedge M - \varepsilon < \uptau_z(\omega)$, for all $s \in [0 , \uptau_z(\omega) \wedge M -\varepsilon]$ we have $X_s^z(\omega) \in D^{o}$. Applying \eqref{EL1.4A} we get
$X^n_s\in D^{o}$ for all $s \in [0 , \uptau_z(\omega) \wedge M - \varepsilon]$
and for all large $n$.
Thus $\uptau_z \wedge M - \varepsilon \leq \uptau_n \wedge M$ for all large $n$. Thus for every $\varepsilon>0$ and $\omega\in \Omega_M\cap\tilde\Omega$, we have $N(\omega)$ satisfying
$$
\uptau_z(\omega) \wedge M - \varepsilon \leq \uptau_n(\omega) \wedge M \leq \uptau_z(\omega) \wedge M + \varepsilon
$$
for all $n\geq N(\omega)$.   Hence we proved 
$M\wedge\uptau_n \to M\wedge\uptau_z$ pointwise in
$\omega\in \Omega_M\cap\tilde\Omega$, as $n \to \infty$. Since $M$ is arbitrary, this also implies 
$\uptau_n \to \uptau_z$ almost surely.

Next we want to show that exit location converges, that is, 
$ X^n_{\uptau_n} \to X^z_{\uptau_z}$ as $n \to \infty$, almost surely.
We recall 
the L\'evy system formula ( cf. \cite[p.~65]{BSW}),

\begin{equation}\label{EL1.4B}
\Exp_x \left[\sum_{0<s \leq t} f(X_{s{-}} , X_s) \Ind_{\{ X_{s{-}} \neq X_s \}} \right] = \Exp_x \left[\int_0^t \int_{\Rd} f(X_{s} , y) 
j(y-X_s)\D{y}\, \D{s} \right]
\end{equation}
which holds for all non-negative $f \in \cB_b(\Rd \times \Rd)$.  Now we put $f_{\varepsilon} (x,y) = \Ind_{\{ x\in D\}} 
\Ind_{\{ \vert x-y \vert > \varepsilon \}}  \Ind_{\{ y \in \partial D\}}$,  in the above formula and using the fact $\abs{\partial D}=0$ we obtain
$$
\Exp_z \left[\sum_{0<s \leq t} f_{\varepsilon}(X_{s_{-}} , X_s) \Ind_{\left\lbrace X_{s{-}} \neq X_s \right\rbrace } \right] = 0,
\quad \text{for all}\; \varepsilon>0.
$$
Since $\varepsilon$ and $t$ are arbitrary, this gives us
\begin{equation}\label{EL1.4C}
\Prob_z (\left\lbrace  X_{\uptau_z} \in \partial D ,  X_{\uptau_z-} \neq X_{\uptau_z} \right\rbrace) =0.
\end{equation}

Again, choosing 
$$\hat{f}_{\varepsilon}(x,y) = \Ind_{\{ x \in \partial D \} } 
\Ind_{\{ y \in D^{+}_{\varepsilon} \} }, \quad \text{for}\quad
D^{+}_{\varepsilon} = \{ y : \dist(y, D) > \varepsilon \}
$$ 
in \eqref{EL1.4B} and since $X_s$ has transition density (see
\cref{L2.2}), it follows
that
$$
\Exp_z \left[\sum_{0<s \leq t} \hat{f}_{\varepsilon}(X_{s{-}} , X_s) \Ind_{\left\lbrace x_{s^{-}} \neq X_s \right\rbrace     } \right] = 0.
$$
Since $\varepsilon, t$ are arbitrary, we get
\begin{equation}\label{EL1.4D}
\Prob_z (\{  X_{\uptau_z^{-}} \in \partial D ,  X_{\uptau_z} \in (\bar{D})^c \}) =0.
\end{equation}
We claim that for any $M>0$ we have
$$
X^n_{\uptau_n \wedge M} \to X^z_{\uptau_z \wedge M} \quad \text{as} \quad n \to \infty,\; \text{almost surely}.
$$
We will be interested in the case where $\uptau_z < M$ since,  
given $t = M$, function $t\mapsto X_t$ is almost surely continuous at 
$t = M$.  Now fix $\omega$ so that it is in the complement of
the events in \eqref{EL1.4C} and \eqref{EL1.4D}.  Since process $X_t$ is C\`{a}dl\`{a}g
(right continuous with a finite left limit) and $\uptau_n \wedge M \to \uptau_z \wedge M$ almost surely,  
we only  need to consider a situation where $\uptau_n \wedge M \nearrow \uptau_z \wedge M$. On set $\{ \uptau_z < M \}$, we have $\uptau_n  \nearrow \uptau_z $.

If $t \to X^z_t(\omega)$ is continuous at $\uptau_z(\omega) = t$, 
then we have the claim.  So we let $X^z_{\uptau_z{-}}(\omega) \neq X^z_{\uptau_z}(\omega)$.  Since $\omega$ is in the complement of the
events in
$\eqref{EL1.4C}$ and $\eqref{EL1.4D}$, we have $X^z_{\uptau_z{-}}(\omega) \in D^{o} $ and $ X^z_{\uptau_z}(\omega) \in \overline{D}^c$. But $x_n \to z$ and $X^z\mid_{[0,\uptau_z^{-}]} (\omega) $ is in $D^{\circ}$, then $X^n\mid_{[0,\uptau_n]} (\omega)$ is in $D^{\circ}$ for large $n$, contradicting the fact that $X^n_{\uptau_n} (\omega) \in D^c$. So this case can not happen and we get that 
$X^n_{\uptau_n \wedge M} \to X^z_{\uptau_z \wedge M}$ as $n \to \infty$, almost surely. Since $M$ is arbitrary and 
$\uptau_n\to \uptau_z$, it gives us
\begin{equation}\label{EL1.4E}
X^n_{\uptau_n} \to X^z_{\uptau_z} \quad \text{as} \quad n \to \infty,\; \text{almost surely}.
\end{equation}

Now we are ready to show that $u(x_n)\to u(z)$,
Applying dominated convergence theorem and using 
\eqref{EL1.4E} we get
$$\Exp_{x_n} [g(X_{\uptau})] \to \Exp_z[g(X_{\uptau})] .$$ 
Since $\uptau_n \wedge M \to \uptau_z \wedge M $  almost surely and for all $\omega$,  $\sup_{t \in [0,M]} \vert X^n_t - X^z_t \vert  \to 0$ (see \eqref{EL1.4A}), we have 
$$
\int_0^{\uptau_n \wedge M} f(X^n_s) \D{s}   \to  \int_0^{\uptau_z \wedge M} f(X^z_s) \D{s},   
$$
almost surely.  Hence by dominated convergence we have 

$$
\Exp \left[ \int_0^{\uptau_n \wedge M} f(X^n_s) \D{s} \right]   \to \Exp \left[ \int_0^{\uptau_z \wedge M} f(X^z_s) \D{s}    \right] .
$$
Now since $M$ is arbitrary, using \cref{L2.1}, it follows that
$$
\Exp \left[ \int_0^{\uptau_n} f(X^n_s) \D{s}   \right]   \to \Exp \left[ \int_0^{\uptau_z} f(X^z_s) \D{s} \right] .
$$
This completes the proof.

It is standard to show that $u$ is a viscosity solution 
(cf. \cite[Remark~3.2]{BL17b} \cite[Theorem~2.2]{Pardoux98}).
\end{proof}

Now we can complete the proof of \cref{T1.1}.
\begin{proof}[Proof of \cref{T1.1}]
Proof follows from \cref{T5.1,T5.2,L5.3}.
\end{proof}

\subsection*{Acknowledgments}
The authors are indebted to the reviewer for his/her careful reading of the paper and suggestions for improving the presentation.
This research of Anup Biswas was supported in part by 
DST-SERB grant MTR/2018/000028 and a Swarnajayanti fellowship. Mitesh Modasiya is partially supported by CSIR PhD fellowship (File no.
09/936(0200)/2018-EMR-I).


\begin{thebibliography}{10}

\bibitem{AC20}
N. Abatangelo and M. Cozzi:
An elliptic boundary value problem with fractional nonlinearity, \emph{ SIAM Journal on Mathematical Analysis}.  53(3)
 (2021), 3577-3601.

\bibitem{ADEJ}
N. Alibaud, F. del Teso, J. Endal and E.R. Jakobsen:
The Liouville theorem and linear operators satisfying the
maximum principle,
\emph{J. Math. Pures Appl.} 142 (2020) 229--242

\bibitem{Barles12}
G. Barles, E. Chasseigne, A. Ciomaga and C. Imbert:
Lipschitz regularity of solutions for mixed integro-differential equations, \emph{J. Differential Equations}, 252 (2012) 6012--6060.

\bibitem{Barles08}
G. Barles, E. Chasseigne and C. Imbert: 
On the Dirichlet problem for second-order elliptic integro-differential equations, 
\emph{Indiana Univ. Math. J.} 57 (2008), no. 1, 213--246.

\bibitem{Barles08a}
G. Barles and C. Imbert:
Second-order elliptic integro-differential equations: viscosity solutions’ theory revisited, 
\emph{Ann. Inst. H. Poincar\'e Anal. Non Lin\'eaire} 25 (2008), 567--585.

\bibitem{BBG}
M.T. Barlow, R.F. Bass and C. Gui:
The Liouville property and a conjecture of De Giorgi,
\emph{Comm. Pure Appl. Math.} 53(8) (2000), 1007--1038.

\bibitem{BDVV}
S. Biagi, S. Dipierro, E. Valdinoci and E. Vecchi:
A Faber-Krahn inequality for mixed local and nonlocal operators, \emph{Journal d'Analyse Math\`{e}matique}, (2023), 1-43.

\bibitem{BDVV-a}
S. Biagi, S. Dipierro, E. Valdinoci and E. Vecchi:
Mixed local and nonlocal elliptic operators: regularity and maximum principles, \emph{Communications in Partial Differential Equations} 47 (2022), no. 3, 585--629

\bibitem{BVDV} S. Biagi, E. Vecchi, S. Dipierro and E. Valdinoci:
Semilinear elliptic equations involving mixed local and nonlocal operators,
\emph{Proceedings of the Royal Society of Edinburgh Section A: Mathematics},
DOI:10.1017/prm.2020.75

\bibitem{BHM}
H. Berestycki, F. Hamel and R. Monneau: One-dimensional symmetry of bounded entire solutions of some elliptic equations, 
\emph{Duke Math. J.} 103(3) (2000), 375--396.

\bibitem{BNV} H. Berestycki, L. Nirenberg, and S. R. S. Varadhan:
The principal eigenvalue and maximum principle for second order elliptic operators in general domains, Comm. Pure Appl. Math. 47 (1994), no. 1, 47--92.

\bibitem{B20}
A. Biswas: 
Principal eigenvalues of a class of nonlinear integro-differential operators, \emph{J. Differential Equations}
268 (2020), no. 9, 5257--5282.

\bibitem{B21} A Biswas:
Existence and non-existence results for a class of semilinear nonlocal operators with exterior condition,
\emph{Pure and Appl. Func. Anal.} Vol. 6 (2021), 289--308

\bibitem{BJK}
I.H. Biswas, E.R. Jakobsen and K.H. Karlsen:
Viscosity solutions for a system of integro-PDEs and connections to optimal switching and control of jump-diffusion processes, 
\emph{Appl. Math. Optim.} 62 (2010), no. 1, 47--80.


\bibitem{BL17b}
A. Biswas and J. L\H{o}rinczi:
Maximum principles and Aleksandrov-Bakelman-Pucci type estimates for non-local Schr\"{o}dinger equations with
exterior conditions, \emph{SIAM J. Math. Anal.} 51 (2019), no. 3, 1543--1581.

\bibitem{BL21}
A. Biswas and J. L\H{o}rinczi:
Hopf's lemma for viscosity solutions to a class of non-local equations with applications,
\emph{Nonlinear Analysis} 204 (2021), 112194  

\bibitem{BMS22} A.Biswas, M. Modasiya and A. Sen:  Boundary regularity of mixed local-nonlocal operators and its application. \emph{Annali di Matematica Pura ed Applicata} (1923-) 202, no. 2 (2023): 679-710.

%


\bibitem{BSW}
B. B\"{o}ttcher, R. Schilling and J. Wang:
L\'evy matters. III.
L\'evy-type processes: construction, approximation and sample path properties. With a short biography of Paul L\'evy by Jean Jacod. Lecture Notes in Mathematics, 2099. L\'evy Matters. Springer, Cham, 2013. xviii+199 pp.

\bibitem{BLL}
H. J. Brascamp, E. H. Lieb and J. M. Luttinger:
A general rearrangement inequality for multiple integrals, J. Funct. Anal. 17 (1974), 227--237.


\bibitem{Cab} X. Cabr\'e: 
On the Alexandroff-Bakel'man-Pucci estimate and the reversed H\"{o}lder inequality for solutions of elliptic and parabolic equations. 
\emph{Comm. Pure Appl. Math.} 48 (1995), no. 5, 539--570

\bibitem{CC}
 L. A. Caffarelli and X. Cabr\'e:
Fully Nonlinear Elliptic Equations, volume 43 of Amer. Math.
Soc. Colloq. Publ. American Mathematical Society, Providence, RI, 1995.

\bibitem{CS09}
L. Caffarelli and L. Silvestre:
Regularity theory for fully nonlinear integro-differential equations,
\emph{Comm. Pure Appl. Math.} 62 (2009), 597--638.

\bibitem{CVW}
D. Cassani, L. Vilasi and Y. Wang:
Local versus nonlocal elliptic equations: short-long range field interactions,
\emph{Advances in Nonlinear Analysis} 10 (2021), no. 1, 895--921

\bibitem{CPSZ}
L. Chen, K. Painter, C. Surulescu and A. Zhigun:
Mathematical models for cell migration: a non-local perspective,
\emph{Phil. Trans. R. Soc. B} 375 (2020), 20190379


\bibitem{Ciomaga} 
A. Ciomaga: On the Strong Maximum Principle for Second Order Nonlinear Parabolic Integro-Differential Equations, \emph{Advances in Diff. Eqns} 17(2012), 635--671.

\bibitem{dEJ}
F. del Teso, J. Endal and E. R. Jakobsen:
On distributional solutions of local and nonlocal problems of porous medium type, \emph{C. R. Math. Acad. Sci. Paris} 355 (2017), no. 11, 1154--1160.

\bibitem{Faber}
G. Faber:
Beweis, dass unter allen homogenen Membranen von gleicher Fl\"{a}che und gleicher Spannung die kreisf\"{o}rmige den tiefsten Grundton gibt, Sitzungsber. Bayer. Akad. Wiss.
M\"{u}nchen, Math.-Phys. Kl. (1923), 169--172.

\bibitem{F99}
A. Farina: 
Symmetry for solutions of semilinear elliptic equations in $\RR^N$ and related conjectures, Papers in memory of Ennio De Giorgi. Ricerche Mat. 48 (1999), suppl., 129--154.

\bibitem{FV11} A. Farina and E. Valdinoci:
Rigidity results for elliptic PDEs with uniform limits: an abstract framework with applications, 
\emph{Indiana Univ. Math. J.} 60(1) (2011), 121–141.

\bibitem{FW14}
P. Felmer and W. Wang:
Radial symmetry of positive solutions to equations involving the
fractional Laplacian. 
\emph{Commun. Contemp. Math.} 61 (2014), 1350023

\bibitem{FS08} R.L. Frank and R. Seiringer, Non-linear ground state representations and sharp Hardy inequalities, J. Funct. Anal. 255 (2008), 3407--3430.

\bibitem{GM02}
M.G. Garroni and J.L. Menaldi: 
Second order elliptic integro-differential problems, 
Chapman \& Hall/CRC Research Notes in Mathematics, 430. Chapman \& Hall/CRC, Boca Raton, FL,
2002. xvi+221 pp.

\bibitem{GT99}
G.W. Gibbons and P.K. Townsend:
Bogomol'nyi equation for intersecting domain walls,
\emph{Phys. Rev. Lett.} 83(9) (1999), no. 9, 1727--1730.

\bibitem{GNN}
B. Gidas, W. M. Ni and L. Nirenberg:
Symmetry and related properties via the maximum principle, 
\emph{Comm. Math. Phys.} 68 (1979) 209--243.

\bibitem{Hen06}
A. Henrot:
Extremum problems for eigenvalues of elliptic operators. 
Frontiers in Mathematics. Birkh\"{a}userVerlag, Basel, 2006.


\bibitem{KK21}
T. Klimsiak and T. Komorowski:
Hopf type lemmas for subsolutions of integro-differential equations,
\emph{Bernoulli} 29(2) (2023), 1435-1463.


\bibitem{Krahn}
E. Krahn: \"{U}ber Minimaleigenschaften der Kugel in drei und mehr Dimensionen,
\emph{Acta Comm. Univ. Tartu} (Dorpat) A9, (1926), 1--44.


\bibitem{Millar} P.W. Millar:
First passage distributions of processes with independent increments,
\emph{Ann. Probab.} 3, No. 2 (1975), 215--233.

\bibitem{Mou19} C. Mou:
 Existence of $C^\alpha$ solutions to integro-PDEs,
\emph{Calc. Var. Partial Differential Equations} 58 (2019), no. 4, Paper No. 143, 28 pp.

\bibitem{MS18} C. Mou and A. \'{S}wi\k{e}ch:
Aleksandrov-Bakelman-Pucci maximum principles for a class of uniformly elliptic and parabolic integro-PDE, J. Differential Equations 264 (2018), no. 4, 2708--2736.

\bibitem{N12} 
T. Nagylaki:
Clines with partial panmixia,
\emph{Theor Popul Biol.} 81(1)(2012), 45--68.

\bibitem{OS07}
B. \O ksendal and A. Sulem:
\emph{Applied stochastic control of jump diffusions},
Second edition. Universitext. Springer, Berlin, 2007.

\bibitem{Pardoux98}
E. Pardoux: 
Backward Stochastic differential equations and viscosity solutions of systems of semilinear parabolic and elliptic PDEs of second order. In: Decreuefond, L., Gjerd, J., {\O}ksendal, B., Ust\"{u}nel, A.S. (Eds.), Progr. Probab., 42. Birkh\"{a}user Boston, Boston, MA. 79–127 (1998)

\bibitem{RS15}
X. Ros-Oton and J. Serra:
Nonexistence results for nonlocal equations with critical and supercritical nonlinearities, 
\emph{Comm. Partial Differential Equations} 40 (2015), 115--133.

\bibitem{QSX}
A. Quaas, A. Salort and A. Xia:
Principal eigenvalues of fully nonlinear integro-differential elliptic equations with a drift term, \emph{ESAIM Control Optim. Calc. Var.} 26 (2020), Paper No. 36, 19.

\bibitem{Sato} K.-I. Sato: L\'{e}vy Processes and Infinitely Divisible Distributions, Cambridge University Press, 1999.

\bibitem{Serrin}
J. Serrin:
A symmetry problem in potential theory,
\emph{Arch. Ration. Mech. Anal.} 43 (1971) 304--318.

\bibitem{Sch98}
R. L. Schilling:
Growth and H\"{o}lder conditions for the sample paths of Feller processes. \emph{Probab. Theory Related Fields} 112, no. 4 (1998), 565--611


\bibitem{Sztonyk}
P. Sztonyk:
On harmonic measure for L\'{e}vy processes, Probab. Math. Statist. 20 (2000), 383-390.

\bibitem{W83} T. Watanabe:
The isoperimetric inequality for isotropic unimodal L\'evy processes,
\emph{Zeit. Wahrsch. Verw. Gebiete} 63 (1983), 487--499.


\end{thebibliography}
\end{document}